\documentclass[plain]{amsart}
\usepackage[toc,page]{appendix}
\usepackage{amssymb,graphicx,epsfig,amsxtra, amsmath}
\usepackage[]{graphics}
\usepackage{amscd} 
\usepackage{xypic}
\usepackage[mathscr]{eucal}
\usepackage{xcolor}
\usepackage{a4}
\usepackage{ulem}
\usepackage{paralist}
\usepackage{comment}
\input{xy}
\xyoption{all}
\newtheorem{theorem}{\textbf{Theorem}}[section]
\newtheorem{proposition}[theorem]{\textbf{Proposition}}
\newtheorem{lemma}[theorem]{\textbf{Lemma}}
\newtheorem{corollary}[theorem]{\textbf{Corollary}}

\newtheorem{claim}[theorem]{\textbf{Claim}}
\newtheorem{problem}[theorem]{{Problem}}
\theoremstyle{definition}
\newtheorem{definition}[theorem]{Definition}
\newtheorem{remark}[theorem]{Remark}

\newtheorem{definition-remark}{Definition-Remark}

\def\ag{\`a}

\def\Sing{\operatorname{Sing}}

\def\geq{\geqslant}
\def\leq{\leqslant}

\begin{document}

\title[Limits of nodal surfaces and applications]{Limits of nodal surfaces and applications }
\author
{Ciro Ciliberto}
\address{Dipartimento di Matematica\\
Universit\`a di Roma ``Tor Vergata''\\
 Via della Ricerca Scientifica\\
00133 Roma, Italy.}
\email{cilibert@mat.uniroma2.it}

\author{Concettina Galati }
\address{Dipartimento di Matematica\\
 Universit\ag\, della Calabria\\
via P. Bucci, cubo 31B\\
87036 Arcavacata di Rende (CS), Italy. }
\email{concettina.galati@unical.it }




\date{February 26th, 2025}

\dedicatory{}

\commby{}


\begin{abstract}  Let  $\mathcal X\to\mathbb D$ be a  flat family of projective complex 3-folds  over a disc $\mathbb D$  with smooth  total space $\mathcal X$ and smooth general fibre  $\mathcal X_t,$  and whose special fiber $\mathcal X_0$  has double  normal crossing singularities, in particular,  $\mathcal X_0=A\cup B$, with $A$, $B$ smooth  threefolds intersecting transversally along a  smooth surface $R=A\cap B.$ In this paper we first study the limit singularities of a $\delta$--nodal surface in the general fibre $S_t\subset\mathcal X_t$, when   $S_t$   tends to the central fibre  in such a way its  $\delta$ nodes  tend to distinct points in $R$. The result  is that the limit surface $S_0$ is in general the union $S_0=S_A\cup S_B$, with $S_A\subset A$, $S_B\subset B$ smooth surfaces, intersecting on $R$ along a $\delta$-nodal curve $C=S_A\cap R=S_B\cap B$. Then we prove that, under suitable conditions, a surface $S_0=S_A\cup S_B$ as above indeed deforms to a $\delta$--nodal surface in the general fibre of 
$\mathcal X\to\mathbb D$.  As applications we prove that there are regular irreducible components of the Severi variety  of  degree $d$  surfaces  with $\delta$ nodes   in $\mathbb P^3$, for every $\delta\leq {d-1\choose 2}$ and of the Severi variety of complete intersection  $\delta$-nodal  surfaces  of type $(d,h)$,  with $d\geq h-1$  in $\mathbb P^4$,  for every  $\delta\leq {{d+3}\choose 3}-{{d-h+1}\choose 3}-1.$

\end{abstract}


\maketitle

\tableofcontents

\section{Introduction}

The main object of study in this article is Severi varieties of nodal surfaces on smooth, projective, complex threefolds. 
Severi varieties of nodal hypersurfaces on a smooth variety are a well know object of study in algebraic geometry, that goes back to well more than a century ago. Its importance is underlined by the relationships with other themes in the area. For example, the recent papers \cite {DF, Th} explore the relation of Severi varieties with the Hodge conjecture.  

Our approach to the subject is via degenerations. Degenerations of smooth complex  varieties to  complex varieties with simple normal crossings is also a classical object of study. In particular it has been widely used by several authors for studying Severi varieties of nodal curves on  surfaces. The method is powerful and enables one to obtain sharp results on the non--emptiness of some Severi varieties of curves (see, for instance, \cite {Ch, CDGK1, CDGK2, GaK}, etc.). 

One of the basic ideas in these papers is the well known and classical fact that 
the limit of a curve $C_t$ with a node $p_t$ on a smooth surface $\mathcal X_t$, when $\mathcal X_t$ degenerates to a reducible surface $\mathcal X_0=A\cup B$, with $A$ and $B$ smooth and meeting transversally along a smooth curve 
$R=A\cap B$, and $p_t$ going to a  point $p_0\in R$, is a curve $C_0\subset\mathcal X_0$ with a tacnode in $p_0$, which appears scheme theoretically with multiplicity $2$. This result is an easy consequence of the study of the versal deformation space of a tacnode, and its proof is in \cite {CH,R}. This result has been proved also using limit linear systems techniques, see \cite{galati}. The present article intends to extend this result on the limit of a nodal curve to the case of nodal surfaces in threefolds, and we will take the point of view of \cite{galati}.  In the sequel, a {\it node} of a surface will be an $A_1$-singularity.

Let  $\mathcal X\to\mathbb D$ be a  flat family of projective complex 3-folds  over a disc $\mathbb D$  with smooth  total space $\mathcal X$ and smooth general fibre, and whose special fiber $\mathcal X_0$  has double  normal crossing singularities, in particular,  $\mathcal X_0=A\cup B$, with $A$, $B$ smooth threefolds intersecting transversally along a  smooth surface $R=A\cap B.$ 

First of all we will study in Section \ref {sect: getting triple point}  the limit singularities of a $\delta$--nodal surface in the general fibre  $S_t\subset\mathcal X_t$, when  $S_t$   tends to the central fibre   in such a way that its  $\delta$ nodes  tend to distinct points  $p_1,\ldots,p_\delta$  in $R$. The result (see Theorem \ref {thm-node}) is that the limit surface $S_0$ is in general the union $S_0=S_A\cup S_B$, with $S_A\subset A$, $S_B\subset B$ smooth surfaces, that cut out on $R$ the same curve $C$   having nodes at $p_1,...,p_\delta$ and no further singularities.  In this case we say that $S_0$ presents  a singularity  of type $T_1$ at   every point  $p_i$, $i=1,...,\delta$. The equations of a $T_1$ singularity are given in \eqref{eq:tacnode}.   
Finally in \ref {ssec:loc} we provide the local equation of (an example of) a local deformation of a singularity of type $T_1$ to a node on the general fibre. 

The central part of our paper is Section \ref{deformations}.  First of all we prove in Lemma \ref {lm:solo-nodo} that 
the only singularity of a surface $S_t\subset \mathcal X_t$ to which a singularity of type $T_1$  of a surface $\mathcal S_0\subset\mathcal X_0$ may be deformed is a node.  In \S \ref {ssec:defo} we describe the 
 first order  locally trivial deformations in $\mathcal X_0$ of surfaces $S_0=S_A\cup S_B$ with $T_1$ singularities on $R$ and at most nodes elsewhere. In particular we find sufficient conditions for smoothness of the equisingular deformation locus of $S_0$ in the relative Hilbert scheme of $\mathcal X$. If these conditions are verified, then the $T_1$ singularities of $S_0$ and its nodes can be smoothed independently inside $\mathcal X_0$.  Next, in \S \ref {sect:smoothing-to-nodes}, we consider deformations of a surface $S_0\subset\mathcal X_0$, with $T_1$ singularities on $R$ and at most nodes elsewhere, off the central fibre. We prove, in Theorem \ref {thm:main-theorem}, that under suitable conditions, one can deform $S_0$ off the central fibre $\mathcal X_0$ to a surface $S_t$ in the general fibre $\mathcal X_t$, with only nodes, that are the deformations of the nodes of $S_0$ and of the $T_1$ singularites of $S_0$, and that the space of  this deformation  is generically smooth of the expected dimension. Again, generic smoothness means that the nodes of the general surface $S_t$ can be independently smoothed. 

In Section \ref {sec:appl} we give a couple of applications of our general result. Essentially we consider the following problem (see Problem \ref {regular}). Let $X$  be a smooth irreducible projective complex threefold. Let $L$ be a  line bundle on  $X$ 
such that the general surface in the linear system $| L|$ is smooth and irreducible.
Let $V^{X, |L|}_\delta$ be the Severi variety of surfaces  $S$ in $|L|$ which are reduced  with  only $\delta$ nodes as singularities. The question we consider is: given $X$ and $L$ as above, which is the  maximal value  of $\delta$ such that 
 $V^{X, |L|}_\delta$ has a generically smooth component of the expected codimension $\delta$ in $|L|$? We give contributions to this problem in two cases. The first one is for $X=\mathbb P^3$ and $L=\mathcal O_{\mathbb P^3}(d)$ (see  Theorem \ref {thm:reg}), the second one is when $X$ is a general hypersurface of degree $h\geq 2$ in $\mathbb P^4$ and $L= \mathcal O_X(d)$ with $d\geq h-1$ (see  Theorem \ref {thm:regol}). 
 
 To finish this introduction it is worth mentioning that the basic idea of a singularity of type $T_1$ being a limit of a node, is already contained, although in a rather obscure form, in B. Segre's paper \cite {Seg}. In this paper Segre considers, even more generally, the case of higher dimension. As a matter of fact, we believe that there should no obstruction in extending our results in higher dimension too. However we did not dwell on this here, because we thought that the surface in threefold case already shows the complexity of the situation. We plan to come back on this in the future. \medskip
 
{\bf Notation}: in what follows we use standard notation in algebraic geometry. In particular, we will denote by $\sim$ the linear equivalence. \medskip


\section{Limit singularity of a node of a surface in a threefold}\label{sect: getting triple point}

\subsection{The problem}\label{ssec:surf} Let  $\mathcal X\to\mathbb D$ be a  flat family of projective complex 3-folds  over a disc $\mathbb D$  with smooth  total space $\mathcal X$ and smooth general fiber
$\mathcal X_t$, with $t \in \mathbb D\setminus \{0\}$, and whose special fiber $\mathcal X_0$  has double  normal crossing singularities, in particular,  $\mathcal X_0=A\cup B$
has two  smooth  irreducible components $A$ and $B$, intersecting transversally along a 
smooth surface $R=A\cap B.$ 

Let $\mathcal L$ be a line bundle on $\mathcal X$. For each $t\in \mathbb D$ we set
$\mathcal L_t= \mathcal L_{| \mathcal X_t}$.  
We consider the following question. Roughly speaking, assume that for $t\in \mathbb D$ general we have a surface $S_t\in |\mathcal L_t|$ having a double point $p_t$. Assume that $S_t$ tends to a surface $S_0$ in $\mathcal X_0$ with $p_t$ tending to a point $p_0\in R$. The question is: what is the  singularity  that $S_0$ has  at $p_0$.  Let us make this setting more precise.  

\subsection{Set up} Let us fix $p=p_0\in R$, which is a double point for the central fibre $\mathcal X_0$ whereas $\mathcal X$ is smooth at $p$. Hence 
there are no sections of $\mathcal X\to\mathbb D$ passing through $p$. So let us consider a smooth bisection $\gamma^\prime$ of $\mathcal X\to\mathbb D$ passing trough $p$. 
\smallskip

{\bf \em Step 0.} Let us look at the following commutative diagram

\begin{displaymath}
\xymatrix{
\mathcal Y\ar[r]\ar[dr]&  \mathcal X^\prime \ar[d]\ar[r] &
\mathcal X\ar[d]\\
& \mathbb D  \ar[r]^{\nu_2}& \mathbb D  }
\end{displaymath}
where the rightmost square is cartesian and   $\nu_2: u\in \mathbb D\to u^2\in \mathbb D$. Then $\mathcal X'$ is singular along the counterimage of $R$ (that by abuse of notation we still denote by $R$), which is a locus of double points  for $\mathcal X'$, with tangent cone a quadric cone of rank 3. The morphism $\mathcal Y\to \mathcal X'$ is the desingularization of $\mathcal X'$ obtained by blowing up $\mathcal X'$ along $R$. 

The induced morphism $\pi:\mathcal Y\to\mathcal X$
is $2:1$ outside the central fibre of $\mathcal Y$. In particular, for every $t\neq 0$ there are exactly two fibres $\mathcal Y_{u_1}$
and $\mathcal Y_{u_2}$ of $\mathcal Y\to\mathbb D$ isomorphic to the fibre $\mathcal X_t$ of $\mathcal X\to\mathbb D$ via $\pi$, where
$\{u_1,\,u_2\}=\nu_2^{-1}(t)$. 
The family  $\mathcal Y\to\mathbb D$ has central  
fibre $\mathcal Y_0=A\cup \mathcal E\cup B$, where, by abusing notation, $A$ and $B$ denote   
the proper transforms of $A$ and $B$  and $\mathcal E\to R$ is a $\mathbb P^1$-bundle 
on $R.$ The morphism $\pi$ is totally ramified along $A$ and $B$ and it contracts $\mathcal E$ to $R$ in $\mathcal X$. In particular, $A\cap\mathcal E$ and $B\cap\mathcal E$ are two non--intersecting sections of $\mathcal E$ 
both isomorphic to $R$. Denote by $F$ the fibre of $\mathcal E\to R$ over the point $p\in R\subset \mathcal X_0$.  One has  $F\cong \mathbb P^1$. Now the counterimage of $\gamma^\prime$ on $\mathcal Y$ is the union of two sections of $\mathcal Y\to \mathbb D$,  
each intersecting $\mathcal Y_0$ at a smooth point on $F$.  We let $\gamma$ be one of these two sections
and $q$ be the intersection point of $\gamma$ and $F$.

Assume there exists an effective divisor $\mathcal S\subset \mathcal Y$,  with $\mathcal S\sim \pi^*(\mathcal L)$ , having double points along $\gamma.$ Let $S$ be the image of $\mathcal S$ in $\mathcal X$ via the morphism $\pi$.  Note that $S$  has points of multiplicity $2$ along the  bisection $\gamma'$.
For every $t\neq 0$, if $\mathcal Y_{u_1}$ and $\mathcal Y_{u_2}$,  with $u_1^2=u_2^2=t$, are the two fibres of $\mathcal Y\to\mathbb D$ isomorphic to $\mathcal X_t$
via $\pi$, we have  
$$S_t=S\cap\mathcal X_t = S_{u_1}\cup S_{u_2},$$
where $S_{u_i}=\pi(\mathcal S_{u_i})$  and $\mathcal S_{u_i}=\mathcal S\cap\mathcal Y_{u_i}$, for  $i=1,2$.  If $t=0$, we have that  
$$
S\cap\mathcal X_0=2S_0=2(S_A\cup S_B),
$$
where  $S_A=\pi(\mathcal S\cap A)\subset A$ and $S_B=\pi(\mathcal S \cap B)\subset B$.

We want to understand  $S|_{\mathcal X_0}.$ To do this, we will first understand $\mathcal S|_{\mathcal Y_0}.$\smallskip

{\bf \em Step 1.} Let $\pi_1:\mathcal Y^1\to\mathcal Y$ be the blowing-up of $\mathcal Y$ along $\gamma$ 
with  exceptional divisor $\Gamma$. We have a new family $\mathcal Y^1\to \mathbb D$ with general fibre the blow up of  $ \mathcal Y_u\cong \mathcal X_{\nu_2(u)}$  at its intersection  point with $\gamma$ (that is also the point of multiplicity $2$ of the surface $\mathcal S_u$), and 
central fibre $\mathcal Y^1_0=A\cup \mathcal E^\prime\cup B,$
where $\mathcal E^\prime$ is the blow-up of $\mathcal E$ at $q.$
Still denoting by $F$ the proper transform of $F$ in $\mathcal Y^1,$ 
we have that the proper transform $\mathcal S^1$ of $\mathcal S$ in $\mathcal Y^1$ satisfies 
\begin{equation}\label{step1}
\mathcal S^1\sim \pi_1^*(\mathcal S)-2\Gamma.  
\end{equation}
We deduce that $\mathcal  S^1 \cdot F=-2$ and hence $F\subset \mathcal  S^1.$ \smallskip

{\bf \em Step 2.} Let now $\pi_2:\mathcal Y^2\to\mathcal Y^1$ be the blow-up of $\mathcal Y^1$ along $F$ with new exceptional divisor $\Theta$.  We have the new family $\mathcal Y^2\to\mathbb D$, whose general fibre is the same as the general fibre of $\mathcal Y^1\to \mathbb D$, 
and new central fibre $\mathcal Y^2_0=A^\prime\cup \mathcal E^{\prime\prime}\cup\Theta\cup B^\prime,$
where $A^\prime,\,\mathcal E^{\prime\prime}$ and $B^\prime$ are the blow-ups of $A,\,\mathcal E^{\prime}$ and 
$B$ at $F\cap A$, $F\subset \mathcal E^\prime$ and $B\cap F$ respectively.
Notice that $\Theta\to F$ is a $\mathbb P^2$-bundle on $F$, intersecting
$A^\prime$ (resp. $B^\prime$) along a surface isomorphic to $\mathbb P^2$, which is a fibre of $\Theta\to F$, and at the same time is the exceptional divisor 
of the blow--up $A^\prime \to A$ at $F\cap A$ (resp. of the blow--up $B^\prime \to B$ at $F\cap B$). Moreover the surface $E:=\Theta\cap \mathcal E^{\prime\prime}$ has a $\mathbb P^1$-bundle structure $E\to F$ and it is  the 
exceptional divisor of  $\mathcal E^{\prime\prime}$, arising from the blowing-up of $F$ in $\mathcal E^{\prime}$. 

We claim that
$E\simeq \mathbb F_0.$  Indeed, since $F\simeq \mathbb P^1$ and $F$ is a fibre of $\mathcal E\to R$, we have that  $\mathcal N_{F|\mathcal E}\simeq\mathcal O_{\mathbb P^1}\oplus\mathcal O_{\mathbb P^1}$. This implies that
$\mathcal N_{F|\mathcal E'}=\mathcal O_{\mathbb P^1}(-1)\oplus\mathcal O_{\mathbb P^1}(-1),$ and hence $E=\mathbb P(\mathcal N_{F|\mathcal E'})=\mathbb F_0.$ 

If $\mathcal S^2$ is the proper transform of $\mathcal S^1$
in $\mathcal Y^2,$ by \eqref{step1}, we deduce that 
 
\begin{eqnarray}\label{step2}
\mathcal S^2|_{\Theta} &\sim & \pi^*_2(\mathcal S^1)|_{\Theta}-m_F\Theta|_{\Theta} \nonumber\\
&\sim &-  2  f_\Theta+m_F(A^\prime+B^\prime+\mathcal E^{\prime\prime})|_{\Theta} \nonumber\\
&\sim & -  2  f_\Theta+m_F(2f_\Theta+ \mathcal E^{\prime\prime}|_\Theta) \nonumber\\
&\sim & (2m_F-  2  )f_\Theta+m_F( \mathcal E^{\prime\prime}|_\Theta)\nonumber\\
&\sim & (2m_F-  2  )f_\Theta+m_FE,
\end{eqnarray}
where $f_\Theta$  denotes the linear equivalence class of a fibre of  $\Theta\to F$ and $m_F$ is the multiplicity of $\mathcal S^1$ along $F.$
Notice that  $\mathcal S^2|_{\Theta}$ must be an effective divisor because it is the restriction to $\Theta$  of an effective divisor  that does not contain $\Theta$. This implies the minimum value of $m_F$ making $\mathcal S^2|_{\Theta}$  effective  is $m_F=1$.

\subsection{Description of $\mathcal S_{|\mathcal Y_0}$ and of $S_{|\mathcal X_0}$} We assume now $m_F=1$. To better understand $\mathcal S^2|_{\Theta}\sim E$, 
we restrict $\mathcal S^2|_{\Theta}$ to $\mathcal E''$. Let $\sigma$ and $f$, with $\sigma^2=f^2=0$, be the
generators of the Picard group of   $E=\mathcal E''\cap\Theta\cong \mathbb F_0$.  By restricting \eqref{step2} to $E$ one gets
\begin{equation}\label{step2-ristretto}
 \mathcal S^2|_E=\mathcal S^2|_{\Theta\cap\mathcal E''}\sim  \mathcal E^{\prime\prime}_{|E}.
\end{equation} 
To compute $\mathcal E^{\prime\prime}_{|E}$,  we use the obvious relation $( A'+B' +\Theta+\mathcal E'')_{|E}=0$, which implies the following identity on $E$
\begin{equation}\label{tr-pt-step2}
2f+\Theta_{|E} +\mathcal E^{\prime\prime}_{|E}=0.
\end{equation}
Since $E=\mathcal E''\cap\Theta$, then $\Theta_{|E}$ is the class of 
$\Theta^2\cdot \mathcal E^{\prime\prime}$ which is clearly the class of  the normal bundle $\mathcal N_{E|\mathcal E''}$ of $E$ in $\mathcal E^{\prime\prime}$. Similarly $\mathcal E^{\prime\prime}_{|E}$ is the class of 
$\Theta\cdot {\mathcal E^{\prime\prime}}^2=c_1(\mathcal N_{E|\Theta})$. Since $E=\mathbb P(\mathcal N_{F|\mathcal E'})$, denoting 
by  $\pi_E:E\to F$ the natural projection morphism, whose fiber is $f$, and by 
$$
e:= \Theta|_E= \Theta^2\cdot \mathcal E'' =c_1(\mathcal N_{E|\mathcal E''}),
$$
we have that $\mathcal N_{E|\mathcal E''}\subset \pi_E^*(\mathcal N_{F|\mathcal E'})$ is the tautological fibre bundle of $E=\mathbb P(\mathcal N_{F|\mathcal E'})$.
So we get that
\begin{equation}\label{tautological}
f\cdot e=-1
\end{equation}
and 
$$
e^2-c_1(\pi_E^*(\mathcal N_{F|\mathcal E'}))\cdot e+c_2(\pi_E^*(\mathcal N_{F|\mathcal E'}))=0,
$$
(see \cite [p. \,606]{gh}). Now 
 $$
 c_2(\pi_E^*(\mathcal N_{F|\mathcal E'}))=\pi_E^*(c_2(\mathcal N_{F|\mathcal E'}))=0,
 $$
 since $\mathcal N_{F|\mathcal E'}$ is a vector bundle on $F$ and $\dim(F)=1.$
 So $$e^2-c_1(\pi_E^*(\mathcal N_{F|\mathcal E'}))\cdot e=e^2-\pi_E^*(c_1(\mathcal N_{F|\mathcal E'}))\cdot e=e^2-c_1(\mathcal N_{F|\mathcal E'})f\cdot e=e^2+c_1(\mathcal N_{F|\mathcal E'})=0.$$
Thus
 \begin{equation}\label{autoint-e}
e^2=-c_1(\mathcal N_{F|\mathcal E'})=2.
\end{equation}
Set $e=a\sigma+b f$. By \eqref{tautological} and \eqref{autoint-e} one gets $a=-1$ and 
$$
-2b=(-\sigma+bf)^2=e^2=2, \,\,\,{\rm \,\,hence}\,\,\,b=-1.
$$
 Thus, we have 
 \begin{equation*}\label{eq:lopt}
 \Theta_{|E}=c_1(N_{E|\mathcal E''})=-\sigma -f,
 \end{equation*}
 hence by \eqref {tr-pt-step2}, we get 
 \begin{equation}\label{normale-E} 
 \mathcal E^{\prime\prime}_{|E}=c_1(N_{E|\Theta})=\sigma-f.
 \end{equation}

 \begin{remark}\label{non-si-muove}
 From \eqref {normale-E} it follows that the divisor $E$ (which does not move on $\mathcal E''$ being there an exceptional divisor)  does not move in $\Theta$ either, 
since $\mathcal N_{E| \Theta}$ is non-effective. Hence, by \eqref {step2} and $m_F=1$, we have $\mathcal S^2|_E=E$. 
\end{remark}

We are now also able to describe the divisor $\mathcal S^2|_{A'\cap \mathcal E''}\cong 
\mathcal S^2|_{B'\cap \mathcal E''}$. Indeed 
$$
\mathcal S^2|_{A'\cap \mathcal E''}\sim ( \pi_2^*\pi_1^*(\mathcal S)-2\pi_2^*(\Gamma)-\Theta)|_{A'\cap \mathcal E''}
\sim \pi_2^*\pi_1^*(\mathcal S)|_{A'\cap \mathcal E''}-\Theta\cap A'\cap \mathcal E'',$$
and, since $\mathcal S^2|_\Theta=E$ by Remark \ref{non-si-muove},   it contains the $(-1)$-curve  $ \Theta\cap A'\cap \mathcal E''=E\cap A'$  in its base locus with multiplicity $1$.
Thus 
$
\mathcal S^2|_{A'\cap \mathcal E''} = \mathcal D_A\cup (\Theta\cap A'\cap \mathcal E''),
$
where 
$$
\mathcal D_A \sim \pi_2^*(\pi_1^*(\mathcal S))|_{A'\cap \mathcal E''}-2\Theta\cap A'\cap \mathcal E'',
$$
and similarly for $\mathcal S^2|_{B'\cap \mathcal E''}.$  

This analysis implies that: 
\begin{theorem}\label{thm-node}  Let $\mathcal S\subset \mathcal Y$ be an effective Cartier divisor as in {\rm{\bf Step 0}}. Then  the surface $\mathcal S_{|\mathcal Y_0}$ is the union of three surfaces  $\mathcal S_A=\mathcal S\cap A, \mathcal S_B=\mathcal S\cap B$ and $\mathcal S_{\mathcal E}=\mathcal S\cap \mathcal E$,  where $\mathcal S_A$ (resp. $ \mathcal S_B$) intersects $A\cap \mathcal E$ (resp. $B\cap \mathcal E$) along a curve which has a double point at the point $F\cap A$  (resp. $F\cap B$), these two curves are isomorphic, and 
$\mathcal S_{\mathcal E}$ is a $\mathbb P^1$--bundle over any one of them.

Accordingly, $S_{|\mathcal X_0}=2S_0$,  with $S_0\in |\mathcal L_0|$  and $S_0$ is the union of two surfaces $S_A$, $S_B$, respectively isomorphic to $\mathcal S_A, \mathcal S_B$, intersecting along a curve in $R$, that has a double point at $p$. 
\end{theorem}

 \begin{figure}[h]
\noindent
\includegraphics[width=7 cm]{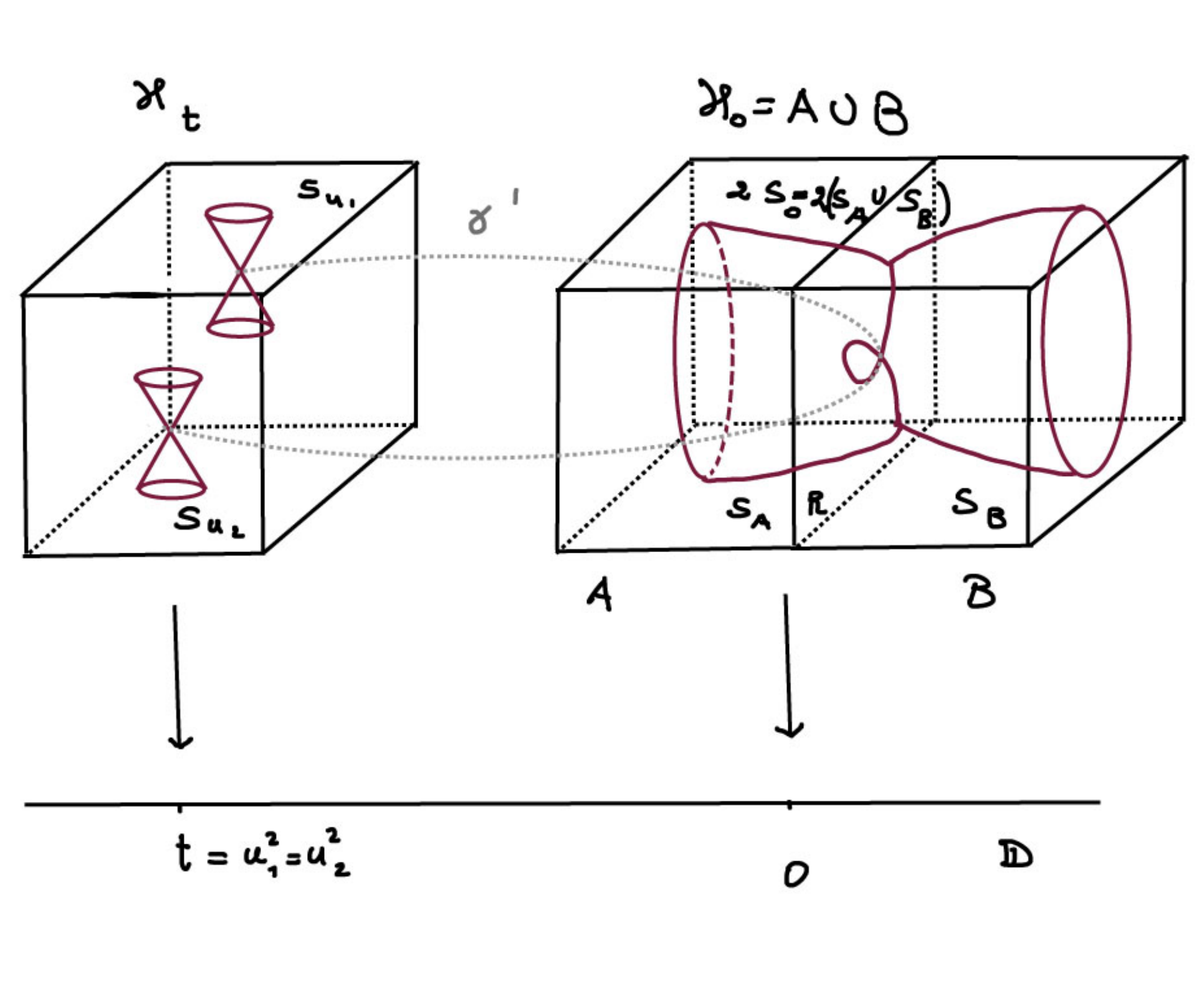}
\caption{}
\end{figure}

\subsubsection{Local equations of  $S_0$ }\label{ssec:node} 
 We may assume that $\mathcal X$ locally around $p\in \mathcal X_0$ is embedded in $\mathbb A^5$ with coordinates $(x,y,z,u,t)$ with $p$ corresponding to the origin.  We may suppose  that $\mathcal X$ is defined by the equation $xy=t$ and the map $\mathcal X\to \mathbb D$ is given by $(x,y,z,u,t)\mapsto t$. So we will assume that $A$ is defined by the equations $x=t=0$ and $B$ by the equations $y=t=0$, so that $R$ is defined by $x=y=t=0$.

The above analysis proves that the surfaces $\mathcal S|_A$ and $\mathcal S|_B$ belong to the restriction linear systems  of $\mathcal L$ to 
$A$ and $B$, respectively, and moreover are tangent to $R$ at the point $p$. 
Thus $S_0=S_A\cup S_B$ belongs to the linear system 
$\mathcal L_0(2,p)\subset |\mathcal L_0|$ of surfaces with local equations at $p$
given by  
\begin{eqnarray}\label{limit-equation-node}
\left\{\begin{array}{l}
(a_1x+b_1y)+f_2(x,y,z,u)=0\\
xy=0,
\end{array}\right.
\end{eqnarray}
with $f_2(x,y,z,u)$ an analytic function with terms of degree at least $2$.

\begin{definition}\label{def:t} Let $S_0=S_A\cup S_B$ be a surface that is the union of two irreducible components $S_A,S_B$ intersecting along a curve $C$. Let $p\in C$. We will say that $S$ has at $p$ a {\it singularity of type $T_1$} if 
$S_A$ and $S_B$ are smooth at $p$ and $C$ has at $p$ a node.
\end{definition}

\begin{remark}\label{rm:general-singularity}   If in \eqref {limit-equation-node}, $a_1b_1\neq 0$, then $S_0$ has a $T_1$ singularity at the origin $p$, and, up to a linear change of coordinates, the local equations  are given by
\begin{eqnarray}\label{eq:tacnode}
\left\{\begin{array}{l}
x+y+f_2(x,y,z,u)=0,\,\,{\rm{with}}\,\, f_2(0,0,z,u)=0\,\,{\rm{having\, a\, node\, at}}\,\,\underline 0,\\
xy=0.\,\, 
\end{array}\right. 
\end{eqnarray}
\end{remark}
 In the sequel we will also refer to $\mathcal L_0(2,p)$ as the sublinear system of $|\mathcal L_0|$ of surfaces with
at least a $T_1$ singularity at $p.$

\begin{remark}\label{condizioni-tacnodo}
  We have that $\mathcal L_0(2,p)\subset|\mathcal L_0|$ has dimension
$$
\dim(\mathcal L_0(2,p))\geq \dim|\mathcal L_0|-3.
$$

\end{remark}

\subsection{Local deformation of a singularity of type $T_1$ to a node}\label{ssec:loc}

In \S \ref {ssec:node} we saw that a singularity of type $T_1$ appears as a {\it generic} limit of a double point of a surface. In this section we will show that locally the converse happens, i.e., that locally a singularity of type $T_1$ can be deformed to a node.   

 In local coordinates $(x,y,z,u,t)$, we consider as before the
family of $3$-folds
$\mathcal X_t:xy=t.$
We further consider the one parameter family of $3$-folds in $\mathbb A^4$ of local equation at $\underline 0$ given by
$$
\mathcal S_{\alpha} : x-y-\alpha(t)-z^2-u^2=0,
$$
where $\alpha(t)$ is a suitable function of $t\in \mathbb A^1$ to be determined, such that 
$\alpha(0)=0$. We will set $S_t=\mathcal S_{\alpha(t)}\cap \mathcal X_t$ for any $t\in \mathbb A^1$.
The surface $S_0$ has a $T_1$ singularity at $\underline 0$ and $\mathcal S_{\alpha}$ is smooth.
Our  requirement on the function $\alpha(t)$ is that for any $t\neq 0$ there exists a singular point $q(t)=(x(t),y(t),z(t),u(t))\in S_t$, i.e.,  such that
$$
T_{q(t)}(\mathcal S_\alpha) = T_{q(t)} (\mathcal X_t). 
$$
This is equivalent to ask that there exists $q(t)=(x(t),y(t),z(t),u(t))\in \mathbb A^4$ satisfying
$$
x(t)-y(t)-\alpha(t)-z(t)^2-u(t)^2  = x(t)y(t) - t = 0
$$
and
$$
(x-x(t))-(y-y(t))-2z(t)(z-z(t))-2u(t)(u-u(t))= c(t)\Big(y(t)(x-x(t))+x(t)(y-y(t))\Big),$$
for a non-zero  $c(t)$. This implies 
$$z(t)=u(t)=0, x(t)=-y(t), \alpha(t)=2x(t)\,\,{\rm and}\,\, t=-x(t)^2=-\frac{{\alpha(t)}^2}{4}.$$
Thus, for every $t\neq 0$ there exist exactly two divisors $\mathcal S_{\alpha_i}$, with $i=1,2$ and
$\alpha_i(t)^2=-4t$ so that 
\begin{eqnarray*}
S_{\alpha_i(t)}=\mathcal S_{\alpha_i}\cap\mathcal Y_t: \left\{\begin{array}{l}
x=y+\alpha_i(t)+z^2+u^2\\
 y(y+\alpha_i(t)+z^2+u^2)=t
 \end{array}\right. 
\end{eqnarray*}
is a one-nodal surface, with tangent cone at $q_i(t)=(\frac{\alpha_i(t)}{2},-\frac{\alpha_i(t)}{2},0,0)$ given by
$$
TC_{q_i(t)}(S_{\alpha_i(t)}):x-y-\alpha_i(t)=2\Big(y+\frac {\alpha_i(t)}2\Big)^2-\alpha_i(t)z^2-\alpha_i(t)u^2=0.
$$

Notice that, for every $i=1,2$, we have that $\alpha_i(t)$ is a well defined continuos function on $\mathbb D_\epsilon^o=\mathbb D(\underline 0,\epsilon)\setminus\{a+i0\,|\,  0<a<\epsilon\}$ (the disk cut along a radius), 
vanishing at $0$ and holomorphic on $\mathbb D_\epsilon^o\setminus \underline0$. Each family $\mathcal S_{\alpha_i}\to \mathbb D_{\epsilon}^o$,  for $i=1,2$, is not algebraic, while 
the complete intersection family of surfaces 
\begin{eqnarray}
\mathcal D: \left\{\begin{array}{l}
(x-y-u^2-z^2)^2=-4t\\
xy=t.
\end{array}\right. 
\end{eqnarray} 
is algebraic. As usual we set $D_t=\mathcal D\cap \mathcal X_t$. One has $D_t=S_{\alpha_1(t)}\cup S_{\alpha_2(t)}$ for $t\neq 0$ and
non--reduced fibre $D_0=2S_0$ for $t=0$.

The  locus $x^2+t=x+y=z=u=0$, whose general point is singular for $D_t$,  is a bisection of $\mathcal X\to\mathbb A^1$  passing through $(\underline 0,0)$.

\section{ Deformations of surfaces with  $T_1$ singularities and nodes} \label{deformations}

 Throughout this section we will consider $\mathcal X\to\mathbb D$ a family of projective complex 3-folds 
over a disc $\mathbb D$ as in the previous section and we let $\mathcal H^{\mathcal X|\mathbb D}$ be its relative Hilbert scheme,
 whose fiber over $t\in\mathbb D$ is the Hilbert scheme of $\mathcal X_t$ and it is denoted by $\mathcal H^{\mathcal X_t}$. Moreover we will consider $S_0=S_A\cup S_B\subset\mathcal X_0$, with $S_A\subset A$ and $S_B\subset B$ an effective reduced Cartier divisor.

\subsection{Deformations and a smoothness criterion}

\begin{definition}\label{def: deformation}
Let $S_0=S_A\cup S_B\subset\mathcal X_0$, with $S_A\subset A$ and $S_B\subset B$ be an effective reduced Cartier divisor and let 
$$\mathcal H_{[S_0]}^{\mathcal X|\mathbb D}\to\mathbb D$$
be an irreducible component of the relative Hilbert scheme of $\mathcal X$ containing $[S_0]$. 
{\it A deformation of $S_0$ in $\mathcal X$ not in $\mathcal X_0$} is the total space $S\subset\mathcal X$
of an irreducible local $r$-multisection $\gamma$ of $\mathcal H_{[S_0]}^{\mathcal X|\mathbb D}$ passing through $[S_0]$. 
Equivalently, a deformation of $S_0$ is an effective divisor 
\begin{displaymath}
\xymatrix{
S\,\,\,\ar[dr] \ar@{^{(}->}[r]&  
\mathcal X\ar[d]\\
& \mathbb D }
\end{displaymath}
 dominating $\mathbb D$, whose central fibre is $S\cap\mathcal X_0=rS_0$, i.e., the surface $S_0$ counted with multiplicity $r$, and whose general fibre 
is a reduced surface with $r$ irreducible components $S\cap\mathcal X_t=S_t^1\cup \cdots \cup S_t^r$, with $[S_t^i]\in\mathcal H_{[S_0]}^{\mathcal X|\mathbb D}$,
for every $i=1,\ldots, t$. We will also say that every irreducible component $S_t^i$ of $S\cap\mathcal X_t$ is a deformation of $S_0$ or that $S_0$ is a limit of $S_t^i$.  Let $\mathcal Y$ be the smooth family of threefolds obtained from $\mathcal X\to\mathbb D$ after a base change 
\begin{displaymath}
\xymatrix{
\mathcal Y\ar[r]\ar[dr]&  \mathcal X^\prime \ar[d]\ar[r] &
\mathcal X\ar[d]\\
&  \mathbb D \ar[r]^{\nu_r}& \mathbb D}
\end{displaymath}
of order $r$ and after minimally desingularizing  the total space of the obtained family. $\mathcal Y$ has central fibre $\mathcal Y_0=A\cup\mathcal E_1\cup\dots\cup\mathcal E_{r-1}\cup B$ with normal crossing singularities of multiplicity two, where every $\mathcal E_i$ is a $\mathbb P^1$-bundle
over $\mathcal E_{i-1}\cap\mathcal E_i$, intersecting $\mathcal E_{i-1}$ and $\mathcal E_{i+1}$, with $\mathcal E_0=A$ and $\mathcal E_{r}=B$. We denote by $\pi:\mathcal Y\to\mathcal X$
the induced morphism. Then the pullback divisor $\pi^*(S)=\mathcal S^1\cup\dots\cup \mathcal S^r$ has $r$ irreducible distinct components
$\mathcal S^1,\dots, \mathcal S^r$, where now every $\mathcal S^i$ has irreducible general fibre
and has central fibre given by $\mathcal S^i_0=\mathcal S^i\cap \mathcal X_0\cong S_0$.  
\end{definition}

\begin{proposition}\label{lm:transverse}
Let $S_0=S_A\cup S_B\subset\mathcal X_0$, with $S_A\subset A$ and $S_B\subset B$, be   a reduced  effective 
Cartier   divisor as above. Let $p$ be a point of the intersection curve
 $C=S_A\cap S_B\subset R$ where $S_A$ and $S_B$ intersect transversally, i.e., such that $S_A$ and $S_B$ are smooth at $p$ and  $T_p(S_A)\neq T_p(S_B).$
Then for every deformation $S\subset\mathcal X$ of $S_0$ not in $\mathcal X_0$, we have that $p$ is limit only of smooth points of the irreducible components of the general fibre of $S$, i.e., in a sufficiently small analytic neighborhood of $p$ in $\mathcal X$, all irreducible components of the general fibre of $S$ are smooth.
In particular, if $S_A$ and $S_B$ intersect transversally along $C$, then $S_0$ is limit only of smooth surfaces.
\end{proposition}
\begin{proof}
Let $S_0=S_A\cup S_B\subset\mathcal X_0$ and $p\in R=S_A\cap S_B$ as in the statement. Then, there exists an analytic coordinate system $(x,y,z,u,t)$ of
$\mathcal X$ at $p=\underline 0$  and such that the local equation of $S_0$ at $p$ is given by $xy=t=z+f_2(x,y,z,u)=0$, where $f_2(x,y,z,u)\in (x,y,z,u)^2$.  

Assume that the assertion is not true. Let $\pi: \mathcal Y\to \mathcal X$ be the morphism defined in  Definition
\ref{def: deformation}, from which we keep the notation.  Then the chain of fibres $F_p^1\cup\dots\cup F_p^{r-1}$ of $\pi^{-1}(S_0)$ contracted to $p$ by $\pi$, intersects  the singular locus of every divisor $\mathcal S^i$. In particular there exists an analytic $s$-multisection $\gamma$ of $\mathcal X$ (with $s\geq r$) passing through $p$, whose general point
is a singular point of an irreducible component of  $S\cap\mathcal X_t$,  with $t$ general.  Every analytic $s$-multisection of $\mathcal X\to\mathbb D$ at $p$ gives rise to $s$ distinct
continuos sections $\gamma^1,\dots,\gamma^s$ over $\mathbb D^o=\mathbb D\setminus\{a+i0\,|\,  0<a<1\}$, which are holomorphic on $\mathbb D^o\setminus\underline 0$.
If $t$ varies in $\mathbb D^o$, then there exists a one-parameter analytic family of irreducible surfaces $\mathcal Z\subset S$, singular along $\gamma^1$,  whose fibre $\mathcal Z_t$ over $t\neq 0$ is an irreducible component of $S\cap\mathcal X_t$ and whose fibre over $0$ is $\mathcal Z_0=S_0$. The equation of $\mathcal Z_t$ in $\mathbb A^4$ with coordinates $(x,y,z,u)$ is given by 
\begin{eqnarray*}
\left\{\begin{array}{l}
p(x,y,z,u;t)=0\\
xy=t,\,\, 
\end{array}\right. 
\end{eqnarray*}
where $p(x,y,z,u;t)=0$ is an analytic function in $(x,y,z,u)$, whose coefficients are continuos functions in the variable $t\in\mathbb D^o$ which are holomorphic on $\mathbb D^o\setminus \underline 0$. If
$$\gamma^1(t)=(x(t),y(t),z(t),u(t)),$$ then, by the hypothesis that the general fibre of $\mathcal Z$ is singular along $\gamma$,
we have that 
$$
p(x,y,z,u;t)=c(t)(y(t)(x-x(t))+x(t)(y-y(t)))+g_2(x-x(t),y-y(t),z-z(t),u-u(t)),
$$
where  $g_2(x-x(t),y-y(t),z-z(t),u-u(t))\in (x-x(t),y-y(t),z-z(t),u-u(t))^2$.

We moreover have that $p(x,y,z,u;t)$ specializes to $p(x,y,z,u;0)=z+f_2(x,y,z,u)\\
=z-z(0)+f_2(x-x(0),y-y(0),z-z(0),u-u(0))$ as $t$ goes to $0$. This is 
not possible. Thus every irreducible component of the general fibre of a deformation $S\subset\mathcal X$ of $S_0$ is smooth in a neighborhood of $p$.\end{proof}

\subsection{Deformations of $T_1$ singularities}

\subsubsection{ Deformations not in $\mathcal X_0$ of surfaces with $T_1$ singularities.}

 In this section we prove that the only singularity of a surface in $\mathcal X_t$, with $t\neq 0$, 
to which a $T_1$ singularity of a surface $S_0\subset\mathcal X_0$ may be deformed is a node. 

\begin{lemma}\label{lm:solo-nodo}
 Let $S_0=S_A\cup S_B\subset\mathcal X_0$ be a reduced effective Cartier divisor, with $S_A\subset A$ and $S_B\subset B$  as above. Let $p$ be a point of the intersection curve $C=S_A\cap S_B\subset R$ where $S_0$ has a $T_1$ singularity. Let $ S\subset\mathcal X$ be a deformation of $S_0$ not in $\mathcal X_0$. Then there exists 
 a sufficiently small analytic neighborhood of $p$ in $\mathcal X$ such that all irreducible components of the general fibre of $ S$ in that neighborhood are smooth or 
  are  $1$-nodal.
\end{lemma}

\begin{proof}
By Proposition \ref{lm:transverse}, if $ S\subset\mathcal X$ is any deformation of $S_0$ not in $\mathcal X_0$, then all irreducible components of the general fibre of $S$
have, in a sufficiently small neighborhood of $p$, only isolated singularities. We  want to prove that if the $T_1$ singularity of $S_0$ at $p$ is limit of an isolated singularity, then this is a node. We argue as in the proof of Proposition \ref{lm:transverse}.

Let $\mathbb D_\epsilon=\mathbb D(\underline 0,\epsilon)\subset\mathbb A^1$ be the open disc with center at the origin and radius $\epsilon$ and
let $\mathbb D_\epsilon^o=\mathbb D(\underline 0,\epsilon)\setminus\{a+i0\,|\,  0<a<\epsilon\}$. We denote by $t=a+ib$ the coordinate on $\mathbb D_\epsilon$
and by $(x,y,z,u)$ the coordinates 
 in $\mathbb A^4$. In $\mathbb A^4\times \mathbb D_\epsilon^o$ we consider a one parameter family of 
$3$-folds
$$
\mathcal S_t: p(x,y,z,u;t)=0,\,\, t\in \mathbb D_\epsilon^o,
$$
where $p(x,y,z,u;t)$ is a polynomial in $x,y,z,u$ whose coefficients are holomorphic functions  on 
$\mathbb D_\epsilon^o\setminus \underline 0$, continuos in $\underline 0$, and the one parameter family of 
$3$-folds
$$
\mathcal X_t: xy=t,\,\, t\in \mathbb D_\epsilon^o.
$$
Assume that the surface
\begin{eqnarray}
S_0=\mathcal S_0\cap \mathcal X_0: \left\{\begin{array}{l}
p(x,y,z,u;0)= x+y+p_2(x,y,z,u)+o(3)=0\\
xy=0\,\, 
\end{array}\right. 
\end{eqnarray}
has a $T_1$ singularity at $\underline 0$, where $p_2(x,y,z,u)$ is the homogeneous part of degree $2$ of  $p(x,y,z,u;0)$, where $o(3)$ is the sum
of terms of degree at least $3$ in $p(x,y,z,u;0)$, and where, by assumption, $p_2(0,0,z,u)$ has non-zero discriminant. 

Assume that, for $t\neq 0$, there exists $q(t)=(x(t), y(t),z(t),w(t);t)\in S_t=\mathcal S_t\cap \mathcal X_t$ specializing to $\underline 0$, as $t$ goes to $\underline 0$ and 
such that $S_t$ has a singular point at $q(t)$. Thus, $\mathcal S_t$ is smooth at $q(t)$ since $\mathcal S_0$ is smooth at $q(0)=\underline 0$ and we have that 
$$
T_{q(t)}(\mathcal S_t)= T_{q(t)}(\mathcal X_t).
$$
In particular there exists a function $c(t)$, which is non-zero if $t\neq 0$, such that
\begin{eqnarray*}
y(t)(x-x(t))+x(t)(y-y(t))& = &c(t)\frac{\partial p}{\partial x}|_{q(t)}(x-x(t))+c(t)\frac{\partial p}{\partial y}|_{q(t)}(y-y(t))\\
& + & c(t)\frac{\partial p}{\partial z}|_{q(t)}(z-z(t))+c(t)\frac{\partial p}{\partial u}|_{q(t)}(u-u(t)),
\end{eqnarray*}
from which we deduce that
\begin{equation}\label{proportionality}
y(t)=c(t)\frac{\partial p}{\partial x}|_{q(t)},\,\, x(t)=c(t)\frac{\partial p}{\partial y}|_{q(t)},
\end{equation}
\begin{equation}\label{vanishing}
\frac{\partial p}{\partial z}|_{q(t)}=0\,\,{\rm and }\,\, \frac{\partial p}{\partial u}|_{q(t)}=0. 
\end{equation}
As $t$ goes to $\underline 0$,  $c(0)=0$, since  $x(t)\neq 0\neq y(t)$ if $t\neq 0$ but
$x(0)=y(0)=0$ and $\frac{\partial p}{\partial x}|_{q(t)}\neq 0\neq \frac{\partial p}{\partial y}|_{q(t)}$ for any $t$ in a neighborhood of $\underline 0$. 
We now write down the local equations 
\begin{equation}\label{laurent1}
\mathcal X_t: y(t)(x-x(t))+x(t)(y-y(t))+(x-x(t))(y-y(t))=0
\end{equation}
of $\mathcal X_t$ at $q(t)$, and the local equation 
\begin{eqnarray}\label{laurent2}
\,\,\,\,\,\, \mathcal S_t & : & \frac{\partial p}{\partial x}|_{q(t)}(x-x(t))+\frac{\partial p}{\partial y}|_{q(t)}(y-y(t))\\
&+ & \frac{\partial p}{\partial x\partial y}|_{q(t)}(x-x(t))(y-y(t))+\frac{\partial p}{\partial x\partial z}|_{q(t)}(x-x(t))(z-z(t))\nonumber \\ 
&+ & \frac{\partial p}{\partial x\partial u}|_{q(t)}(x-x(t))(u-u(t))+\frac{\partial p}{\partial y\partial z}|_{q(t)}(y-y(t))(z-z(t))\nonumber \\
&+ & \frac{\partial p}{\partial y\partial u}|_{q(t)}(y-y(t))(u-u(t))+\frac{\partial p}{\partial u\partial z}|_{q(t)}(u-u(t))(z-z(t))\nonumber \\
&+ & \frac{1}{2}\frac{\partial p}{\partial x^2}|_{q(t)}(x-x(t))^2+\frac{1}{2}\frac{\partial p}{\partial y^2}|_{q(t)}(y-y(t))^2
+ \frac{1}{2} \frac{\partial p}{\partial u^2}|_{q(t)}(u-u(t))^2\nonumber \\
&+ & \frac{1}{2}\frac{\partial p}{\partial z^2}|_{q(t)}(z-z(t))^2 + o(3)=0\nonumber 
\end{eqnarray}
of 
$\mathcal S_t$ at $q(t)$, where $o(3)\in (x-x(t), y-y(t), z-z(t), u-u(t))^3$. By \eqref{laurent1}, one may write
\begin{equation}
x-x(t)=-\frac{x(t)(y-y(t))}{y-y(t)+y(t)}.
\end{equation}
Let $d$ be the maximum degree of $x-x(t)$ in \eqref{laurent2}. By substituting in \eqref{laurent2}, by multiplying by $y^d= (y-y(t)+y(t))^d$,
and by using that $T_{q(t)}(\mathcal S_t)=T_{q(t)}(\mathcal X_t)$, i.e., $y(t)\frac{\partial p}{\partial y}|_{q(t)}=x(t)\frac{\partial p}{\partial x}|_{q(t)}$ for any $t\neq 0$,
we find that the local equation of $S_t=\mathcal S_t\cap \mathcal X_t$ is given by 
\begin{eqnarray}\label{laurent3}
S_t & : & \frac{\partial p}{\partial x}|_{q(t)}\Big(-x(t)(y-y(t))\Big)\Big(y-y(t)+y(t)\Big)^{d-1}\nonumber\\
& + & \Big( \frac{\partial p}{\partial x\partial y}|_{q(t)}(y-y(t))\Big)\Big(-x(t)(y-y(t))\Big)\Big(y-y(t)+y(t)\Big)^{d-1}\nonumber\\
& + & \Big(\frac{\partial p}{\partial x\partial z}|_{q(t)}(z-z(t))\Big)\Big(-x(t)(y-y(t))\Big)\Big(y-y(t)+y(t)\Big)^{d-1}\nonumber \\
& + & \Big(  \frac{\partial p}{\partial x\partial u}|_{q(t)}(u-u(t))\Big)\Big(-x(t)(y-y(t))\Big)\Big(y-y(t)+y(t)\Big)^{d-1}\nonumber \\
&+& \frac{1}{2}\frac{\partial p}{\partial x^2}|_{q(t)}\Big(-x(t)(y-y(t))\Big)^2\Big(y-y(t)+y(t)\Big)^{d-2} \nonumber\\
& + & \frac{\partial p}{\partial y}|_{q(t)}\Big(y-y(t)\Big)\Big(y-y(t)+y(t)\Big)^{d}\nonumber\\
& + &  \Big(\frac{\partial p}{\partial y\partial z}|_{q(t)}(z-z(t)) + \frac{\partial p}{\partial y\partial u}|_{q(t)}(u-u(t))\Big)\Big(y-y(t)\Big)\Big(y-y(t)+y(t)\Big)^d \nonumber \\
& + & \Big(\frac{1}{2}\frac{\partial p}{\partial y^2}|_{q(t)}(y-y(t))^2+\frac{\partial p}{\partial u\partial z}|_{q(t)}(u-u(t))(z-z(t))\Big)\Big(y-y(t)+y(t)\Big)^{d}\nonumber \\
& + & \Big(\frac{1}{2} \frac{\partial p}{\partial u^2}|_{q(t)}(u-u(t))^2 + \frac{1}{2}\frac{\partial p}{\partial z^2}|_{q(t)}(z-z(t))^2\Big)\Big(y-y(t)+y(t)\Big)^d + o(3)\nonumber\\
& = & y(t)^{d-2}\Big(\frac{\partial p}{\partial y}|_{q(t)}y(t)-\frac{\partial p}{\partial x\partial y}|_{q(t)}x(t)y(t)
+\frac{1}{2}\frac{\partial p}{\partial x^2}|_{q(t)}x(t)^2+\frac{1}{2}\frac{\partial p}{\partial y^2}|_{q(t)}y(t)^2\Big)\Big(y-y(t)\Big)^2\nonumber\\
& + & y(t)^{d-1}\Big(y(t)\frac{\partial p}{\partial y\partial z}|_{q(t)}-x(t)\frac{\partial p}{\partial x\partial z}|_{q(t)}\Big)\Big(y-y(t)\Big)\Big(z-z(t)\Big)\nonumber\\
& + & y(t)^{d-1}\Big(y(t)\frac{\partial p}{\partial y\partial u}|_{q(t)}-x(t)\frac{\partial p}{\partial x\partial u}|_{q(t)}\Big)\Big(y-y(t)\Big)\Big(u-u(t)\Big)\nonumber\\
& + & y(t)^{d} \Big(\frac{1}{2}\frac{\partial p}{\partial z^2}|_{q(t)}(z-z(t))^2+\frac{1}{2}\frac{\partial p}{\partial u^2}|_{q(t)}(u-u(t))^2\Big)\nonumber\\
& + & y(t)^{d} \Big(\frac{\partial p}{\partial u\partial z}|_{q(t)}(u-u(t))(z-z(t))\Big)+o(3)=0.\nonumber
\end{eqnarray}
Up to the irrelevant factor $y(t)^{d-2}$, the Hessian matrix at $q(t)$ of the above polynomial 
is
\begin{equation}
H_{q(t)}=\left(\begin{array}{ccc}A_{11}(t) & A_{12}(t) & A_{13}(t) \\
A_{12}(t) & \frac{y(t)^2}{2}\frac{\partial p}{\partial z^2}|_{q(t)} & \frac{y(t)^2}{2}\frac{\partial p}{\partial z\partial u}|_{q(t)} \\
A_{13}(t) & \frac{y(t)^2}{2}\frac{\partial p}{\partial z\partial u}|_{q(t)} & \frac{y(t)^2}{2}\frac{\partial p}{\partial u^2}|_{q(t)}\end{array}\right),
\end{equation} 
where
\begin{eqnarray*}
A_{11}(t) & = & \frac{\partial p}{\partial y}|_{q(t)}y(t)-\frac{\partial p}{\partial x\partial y}|_{q(t)}x(t)y(t)
+\frac{1}{2}\frac{\partial p}{\partial x^2}|_{q(t)}x(t)^2+\frac{1}{2}\frac{\partial p}{\partial y^2}|_{q(t)}y(t)^2,\\
A_{12}(t) & = & \frac{y(t)}{2}\Big(y(t)\frac{\partial p}{\partial y\partial z}|_{q(t)}-x(t)\frac{\partial p}{\partial x\partial z}|_{q(t)}\Big),\\
A_{13}(t) & = & \frac{y(t)}{2}\Big(y(t)\frac{\partial p}{\partial y\partial u}|_{q(t)}-x(t)\frac{\partial p}{\partial x\partial u}|_{q(t)}\Big).
\end{eqnarray*}
Now $S_t$ has a node at $q(t)$ if and anly if $\det(H_{q(t)})\neq 0$. If we substitute the equalities \eqref{proportionality}
in $H_{q(t)}$, we see that this matrix has the first column divisible by $c(t)$ and the second and third columns divisible by $c(t)^2$. 
Let $B_{q(t)}$ be the matrix obtained by $H_{q(t)}$ by dividing the first column by $c(t)$ and the second and third columns by $c(t)^2$. 
We have that
\begin{equation}
B_{q(t)}=\left(\begin{array}{ccc}B_{11}(t) & B_{12}(t) & B_{13}(t) \\
B_{21}(t) & \frac{1}{2}(\frac{\partial p}{\partial x}|_{q(t)})^2\frac{\partial p}{\partial z^2}|_{q(t)} & \frac{1}{2}(\frac{\partial p}{\partial x}|_{q(t)})^2\frac{\partial p}{\partial z\partial u}|_{q(t)} \\
B_{31}(t) &\frac{1}{2}(\frac{\partial p}{\partial x}|_{q(t)})^2\frac{\partial p}{\partial z\partial u}|_{q(t)} & \frac{1}{2}(\frac{\partial p}{\partial x}|_{q(t)})^2\frac{\partial p}{\partial u^2}|_{q(t)}\end{array}\right),
\end{equation} 
where
\begin{eqnarray*}
B_{11}(t) & = & \frac{\partial p}{\partial y}|_{q(t)}\frac{\partial p}{\partial x}|_{q(t)}+ c(t)\Big(-\frac{\partial p}{\partial x\partial y}|_{q(t)}\frac{\partial p}{\partial x}|_{q(t)}\frac{\partial p}{\partial y}|_{q(t)}\\
&+&\frac{1}{2}\frac{\partial p}{\partial x^2}|_{q(t)}\Big(\frac{\partial p}{\partial y}|_{q(t)}\Big)^2+\frac{1}{2}\frac{\partial p}{\partial y^2}|_{q(t)}\Big(\frac{\partial p}{\partial x}|_{q(t)}\Big)^2\Big),\\
B_{21}(t) & = & \frac{1}{2}\frac{\partial p}{\partial x}|_{q(t)}c(t)\Big(\frac{\partial p}{\partial x}|_{q(t)}\frac{\partial p}{\partial y\partial z}|_{q(t)}-\frac{\partial p}{\partial y}|_{q(t)}\frac{\partial p}{\partial x\partial z}|_{q(t)}\Big),\\
B_{31}(t) & = & \frac{1}{2}\frac{\partial p}{\partial x}|_{q(t)}c(t)\Big(\frac{\partial p}{\partial x}|_{q(t)}\frac{\partial p}{\partial y\partial u}|_{q(t)}-\frac{\partial p}{\partial y}|_{q(t)}\frac{\partial p}{\partial x\partial u}|_{q(t)}\Big).
\end{eqnarray*}
As $t$ goes to $0$, $c(t)$ goes to $0$ and the matrix $B_{q(t)}$ specializes to the matrix
\begin{equation}
B_{\underline 0}=\left(\begin{array}{ccc}\frac{\partial p}{\partial y}|_{\underline 0}\frac{\partial p}{\partial x}|_{\underline 0} & B_{12}(0) & B_{13}(0) \\
0 & \frac{1}{2}(\frac{\partial p}{\partial x}|_{\underline 0})^2\frac{\partial p}{\partial z^2}|_{\underline 0} & \frac{1}{2}(\frac{\partial p}{\partial x}|_{\underline 0})^2\frac{\partial p}{\partial z\partial u}|_{\underline 0} \\
0&\frac{1}{2}(\frac{\partial p}{\partial x}|_{\underline 0})^2\frac{\partial p}{\partial z\partial u}|_{\underline 0} & \frac{1}{2}(\frac{\partial p}{\partial x}|_{\underline 0})^2\frac{\partial p}{\partial u^2}|_{\underline 0}\end{array}\right).
\end{equation}
Using that $\frac{\partial p}{\partial x}|_{\underline 0}=1=\frac{\partial p}{\partial y}|_{\underline 0}$, we see that $\det(B_{\underline 0})$ coincides with
the discriminant of the degree $2$ homogeneous polynomial $p_2(0,0,z,u)$, which is non zero by the hypothesis that $S_0$ has a $T_1$ singularity at $\underline 0$.
We finally deduce that $\det(B_{q(t)})\neq 0\neq \det(H_{q(t)})$ and thus the surface $S_t$ has a node at $q(t)$ for $t\neq 0$. 
\end{proof}

\subsubsection{Equisingular deformations of surfaces with $T_1$ singularities}\label {ssec:defo}
 We go on considering the setting we introduced at the beginning of Section \ref{deformations}. 
Assume that $S_0=S_A\cup S_B\subset\mathcal X_0$ is a surface with $S_A$ and $S_B$  smooth, intersecting transversally along $R=A\cap B$, except for
$\delta$ distinct points $p_1,\dots p_\delta\in S_A\cap S_B$, where $S_0$ has a singularity of type $T_1$.
We recall the standard exact sequence 
\begin{equation}\label{eq:standard-sequence}
\xymatrix{
 0  \ar[r] & \Theta_{S_0}  \ar[r] &      \Theta_{\mathcal X}|_{S_0} \ar[r]^{\alpha} & \mathcal N_{S_0|\mathcal X} \ar[r]^{\beta} &   T^1_{S_0}  \ar[r] &     0, }
 \end{equation}
 where $\Theta_{S_0} = \mathfrak{hom}(\Omega_{S_0}^1,\mathcal O_{S_0})$ is the tangent sheaf of $S_0,$
  $\Theta_{\mathcal X}|_{S_0}$ is the tangent sheaf of $\mathcal X$ restricted to $S_0$, $\mathcal N_{S_0|\mathcal X}$ is the normal
   bundle of $S_0$ in $\mathcal X$, and  $T^1_{S_0}$ is the first cotangent sheaf of $S_0$ \cite[Section 1.1.3]{ser}. The latter is supported on the singular 
  locus $\Sing(S_0)=S_A\cap S_B$. The kernel $\mathcal N_{S_0|\mathcal X}^\prime$ of $\beta$ is the so-called 
 {\it equisingular normal sheaf to $S_0$ in $\mathcal X$},
whose global sections are the first order locally trivial deformations of $S_0$ in $\mathcal X$. 
 
In the sequel an equisingular (first order) deformation
of $S_0$ in $\mathcal X$ will be a (first order) locally trivial deformation of $S_0$ in $\mathcal X$. 

We will denoted by
$\mathcal {ES}_{[S_0]}^{\mathcal X_0}\subseteq \mathcal H^{\mathcal X_0}\subset\mathcal H^{\mathcal X|\mathbb D}$ the locally closed  set of equisingular deformations of $S_0$ in $\mathcal X_0$.  Similarly, if $p\in S_0$ is a point, we will denote by
$\mathcal {ES}_{[S_0], p}^{\mathcal X_0}\subseteq \mathcal H^{\mathcal X_0}\subset\mathcal H^{\mathcal X|\mathbb D}$ the locally closed set of deformations of $S_0$
in $\mathcal X_0$, which are equisingular at $p$. 

\begin{lemma}
If $\mathcal {ES}_{[S_0]}^{\mathcal X|\mathbb D}\subseteq \mathcal H^{\mathcal X|\mathbb D}$  is the locally closed  set of equisingular deformations of 
$S_0$ in $\mathcal X$, then $\mathcal {ES}_{[S_0]}^{\mathcal X|\mathbb D}$ coincides set-theoretically with $\mathcal {ES}_{[S_0]}^{\mathcal X_0}\subseteq \mathcal H^{\mathcal X_0}$.
\end{lemma}
\begin{proof}
This is a straightforward consequence of Lemma \ref{lm:solo-nodo}. 
\end{proof}

If $\mathbb T_\delta\subset \mathcal H^{\mathcal X_0}$ is the Zariski closure of the family of surfaces in $\mathcal X_0$ with $\delta$ singularities of type $T_1$, then every irreducible component of $\mathcal {ES}_{[S_0]}^{\mathcal X_0}$ is a Zariski open set in an irreducible component of $\mathbb T_\delta$.

Consider the rational map
$$
\varphi: \mathcal H^{\mathcal X_0}\dasharrow \mathcal H^R
$$
where $\mathcal H^R$ is the Hilbert scheme of $R$, and $\varphi$ maps the general subscheme of $\mathcal X_0$ to its intersection  with $R$. 

\begin{lemma}\label{lm:upper-bound} Let $[S_0]\in \mathbb T_\delta$ be a point corresponding to a surface $S_0=S_A\cup S_B$ as above and suppose that $[S_0]$ is a smooth point of the Hilbert scheme $\mathcal H^{\mathcal X_0}$ so that there is a unique component $\mathcal H^{\mathcal X_0}_{S_0}$ of $\mathcal H^{\mathcal X_0}$
containing $S_0$. Let 
 $C$ be the curve cut out by $S_0$ on $R$. Assume that $h^1(\mathcal N_{C|R}\otimes \mathcal I_{\{p_1,\ldots, p_\delta\}|R})=0$, where $p_1,\ldots, p_\delta$ are the nodes of $C$, which implies that  $\mathcal H^R$ is smooth at the point $[C]$ and that the Severi variety of curves on $R$ with $\delta$ nodes is smooth at the point $[C]$ of codimension $\delta$ in the unique irreducible component $\mathcal H^R_C$ of the Hilbert scheme $\mathcal H^R$ containing $[C]$.  Suppose moreover that the map 
\begin{equation}\label{eq:restr}
 \varphi_{|\mathcal H^{\mathcal X_0}_{S_0}}: \mathcal H^{\mathcal X_0}_{S_0}\dasharrow \mathcal H_C^R
 \end{equation} is dominant. 
  
Then there is an irreducible component $\mathbb T$ of $\mathbb T_\delta$ containing $S_0$ that has codimension at most $\delta$ in  $\mathcal H_{S_0}^{\mathcal X_0}$.
\end{lemma}
\begin{proof} Let $V\subset  \mathcal H_C^R$ be the unique irreducible component of the locally closed set of curves on $R$ with $\delta$ nodes that contains the point $[C]$. Since
$[C]$ sits in the image of  $ \varphi_{|\mathcal H^{\mathcal X_0}_{S_0}}$, $V$ intersects  the image of  $\varphi_{|\mathcal H^{\mathcal X_0}_{S_0}}$.
Since $\varphi_{|\mathcal H^{\mathcal X_0}_{S_0}}$ is dominant, the intersection of $V$ with 
the image of  $ \varphi_{|\mathcal H^{\mathcal X_0}_{S_0}}$ is an open dense subset of $V$, hence
there is an irreducible component  $\mathbb T$ of $\mathbb T_\delta$ containing  $S_0$ such that the map
$$
\varphi_{|\mathbb T}: \mathbb T\dasharrow V
$$ 
is dominant. Let $a$ be the dimension of the general fibre of $\varphi_{|\mathbb T}$ and let $b$ be the dimension of the general fibre of $\varphi_{|\mathcal H^{\mathcal X_0}_{S_0}}$. Of course $a\geq b$.
We have 
$$
\dim(\mathbb T)=\dim (V)+a, \,\, \text{and}\,\, \dim(\mathcal H^{\mathcal X_0}_{S_0})=\dim(\mathcal H_C^R)+b
$$ 
hence
$$
\dim(\mathcal H^{\mathcal X_0}_{S_0})-\dim(\mathbb T)=\dim(\mathcal H_C^R)-\dim (V)+b-a\leq \delta
$$
and the assertion follows. \end{proof}

\begin{remark}\label{rem:eq} Note that in the previous lemma, one has that $\mathbb T$ has exactly codimension $\delta$ in $\mathcal H_{S_0}^{\mathcal X_0}$ if and only if $a=b$.
\end{remark}

\begin{lemma}\label{lm:tangent-space-tacnode} Assume that $S_0=S_A\cup S_B\subset\mathcal X_0$ is a surface with $S_A$ and $S_B$  smooth, intersecting transversally along $R=A\cap B$, except for
$\delta$ points $p_1,\dots, p_\delta\in S_A\cap S_B$, where $S_0$ has a singularity of type $T_1$. 
Then the equisingular first order infinitesimal deformations of $S_0$ in $\mathcal X$ coincide with the equisingular first order  infinitesimal deformations of $S_0$ in $\mathcal X_0$. 
More precisely, we have that 
\begin{equation}\label{eq:equisingular1}
H^0(S_0,\mathcal N_{{S_0}|\mathcal X}^\prime) \subseteq  
H^0(S_0,\mathcal N_{ {S_0}|\mathcal X_0}\otimes I_{\{p_1,\dots,p_\delta\}|\mathcal X_0}),
\end{equation}
where $I_{\{p_1,\dots,p_\delta\}|\mathcal X_0}$ is the ideal sheaf of $\{p_1,\dots,p_\delta\}$ in $\mathcal X_0$. 
\end{lemma}

\begin{proof}
Let $p=p_i$, for  $i=1,...,\delta$, be a point where $S_0$ has a $T_1$ singularity. Consider the localized exact sequence
 \begin{equation}\label{localizationsequenceone}
 \xymatrix{
 0\ar[r] & \mathcal N^\prime_{{S_0}|\mathcal X,\,p}\ar[r] &
\mathcal N_{{S_0}|\mathcal X,\,p}\ar[r] &  T^1_{{S_0},p}\ar[r] & 0.}
 \end{equation} 
 Let $(x,y,z,u,t)$ be an analytic coordinate system of $\mathcal X$ centered at $p$  such that $\mathcal X$ is given by $xy=t$ and
 such that we have the following identifications:
 \begin{itemize}
 \item the local ring $\mathcal O_{{S_0},\,p}=\mathcal O_{\mathcal X,p}/\mathcal I_{{S_0}|\mathcal X,p}$ of $S_0$ at $p$ 
 is identified with \\ $\mathbb C[x,y,z,u]/(h_1,h_2)$, localized at the origin, where $h_1(x,y,z,u)= x+y+h_{12}(x,y,z,u)$,  $h_{12}(x,y,z,u)\in (x,y,z,u)^2$ and $h_{12}(0,0,z,u)=0$ having a node at $\underline 0=p$, and $h_2(x,y,z,u)=xy$; \\
 \item the $\mathcal O_{{S_0},p}$-module $\mathcal N_{{S_0}|\mathcal X,\,p}$ is identified with the free $\mathcal O_{\mathcal X,\,p}$-module \\ 
 $\mathfrak{hom}_{\mathcal O_{\mathcal X,\,p}}(\mathcal I_{{S_0}|\mathcal X ,\,p}, \mathcal O_{S_0,p}),$ generated by the
 morphisms $h_1^*$ and $h_2^*$,  defined by 
 $$h_i^*(s_1(x,y,z,u)h_1(x,y,z,u)+s_2(x,y,z,u)h_2(x,y,z,u))=s_i(x,y,z,u),\,\mbox{for}\, i=1,\,2$$
 \end{itemize}
 and, finally,
 \begin{itemize}
 \item the $\mathcal O_{S_0,p}$-module
 \begin{eqnarray*}
 (\Theta_{\mathcal X}|_{S_0})_{\, p}&\simeq &\Theta_{\mathcal X,p}\otimes \mathcal O_{S_0,p}\\
& \simeq & \langle \partial /{\partial x},\partial /{\partial y},\partial / {\partial z},\partial / {\partial u},\partial /{\partial t} \rangle
 _{\mathcal O_{S_0,\,p}}/
 \langle {\partial}/{\partial t}-x\partial /{\partial y}-y\partial /{\partial x}\rangle
 \end{eqnarray*}
 is identified with  the free $\mathcal O_{\mathcal X,\,p}$-module generated by the derivatives \\
 $\partial /{\partial x},\partial / {\partial y},\partial /{\partial z}, \partial / {\partial u}.$ 
 \end{itemize}
With these identifications, 
the localization $\alpha_p:(\Theta_{\mathcal X}|_{S_0})_{\, p}\rightarrow \mathcal N_{S_0|\mathcal X,\,p}$ of the sheaf map $\alpha$ from \eqref{eq:standard-sequence} is defined by
 \begin{eqnarray*}
\alpha_p(\partial /{\partial x})&=&\,\,\,\,\,\,\Big(s=s_1h_1+s_2h_2\rightarrowtail 
\partial s/\partial x=_{\mathcal O_{S_0,p}}
s_1\partial h_1 /{\partial x}
+s_2\partial h_2 /{\partial x}\Big)\\
&=&\,\,\,\,\,\, (1+\partial h_{12}/\partial x ) h_1^*+yh_2^*,\\
\alpha_p(\partial /{\partial y})&=&\,\,\,\,\,\, (1+\partial h_{12}/\partial y ) h^*_1+xh^*_2,\\
\alpha_p(\partial /{\partial z})&=&\,\,\,\,\,\, (\partial h_{12}/\partial z) h^*_1 \,\,\,\textrm{and}\\
\alpha_p(\partial /{\partial u})&=&\,\,\,\,\,\,(\partial h_{12}/\partial u) h^*_1.
\end{eqnarray*}
By definition of $\mathcal N^\prime_{S_0|\mathcal X}$, a local section $s$ of  $\mathcal N^\prime_{S_0|\mathcal X,\,p}$, is such that there exists a local section $v$ of ${\Theta_{\mathcal X}|_{S_0}}_{\, p}$,
with
$$v=v_x(x,y,z,u)\partial /{\partial x}+v_y(x,y,z,u)\partial /{\partial y}+v_z(x,y,z,u)\partial /{\partial z}+v_u(x,y,z,u)\partial /{\partial u},$$
such that $s=\alpha_p(v)$. Hence, locally at $p$, first order equisingular deformations of $S_0$ in $\mathcal X$ have equations 
\begin{eqnarray}\label{equisingulardeformationtacnode}
\tiny
\left\{\begin{array}{lll}
x+y+h_{12}(x,y,z,u)& + & \epsilon \Big(v_x(1+\partial h_{12}/\partial x)+v_y(1+\partial h_{12}/\partial y)\\
& + &\,\,\,\,\, v_z(\partial h_{12}/\partial z)+v_u(\partial h_{12}/\partial u)\Big)=0\\
xy+\epsilon (yv_x+xv_y) & = & 0.
\end{array}\right.
\end{eqnarray}
The first equation above gives a first order infinitesimal deformation of the Cartier divisor 
 cutting $S_0$ on $\mathcal X_0$, while the second equation 
 gives a first order infinitesimal 
 deformation of $\mathcal X_0$ in $\mathcal X.$ More precisely, by the exact sequence
 
\begin{equation*}\label{eq:standard-sequence-central}
\xymatrix{ 0  \ar[r] &      \Theta_{\mathcal X_0}  \ar[r] &      \Theta_{\mathcal X}|_{\mathcal X_0} \ar[r] & 
\mathcal N_{\mathcal X_0|\mathcal X} \cong\mathcal O_{\mathcal X_0}\ar[r] &   T^1_{\mathcal X_0}\cong\mathcal O_R  \ar[r] &     0, }
 \end{equation*} 
 one sees that $xy+\epsilon (yv_x+xv_y)=0$ is the local equation at $p$ of a first order  equisingular deformation 
of  $\mathcal X_0$ in $\mathcal X$. 
But $H^0(\mathcal X_0, \mathcal N'_{\mathcal X_0|\mathcal X})=H^0(\mathcal X_0, \mathcal I_{R|\mathcal X_0})=0.$ It follows that
the polynomial  $yv_x(x,y,z,u)+xv_y(x,y,z,u)$ in the second equation of  
\eqref{equisingulardeformationtacnode}  must be identically zero, proving the first assertion of the Lemma. 
In particular,  by expanding $v_x$ and $v_y$ in Taylor series, we see that
$$v_x(\underline 0)=v_y(\underline 0)=0.$$

Looking at the first equation of  \eqref{equisingulardeformationtacnode}, we  have that
$\frac{\partial h_{12}}{\partial z}(\underline 0)=\frac{\partial h_{12}}{\partial u}(\underline 0)=0$ since $h_{12}(x,y,z,u)\in (x,y,z,u)^2$. 
This shows the inclusion \eqref{eq:equisingular1}. 
\end{proof}

\begin{remark}\label{rem:gen} The argument in the proof of Lemma \ref{lm:tangent-space-tacnode} proves more than stated. In fact it proves that if $S_0$ is any surface in $\mathcal X_0$ with $\delta$ singularities of type $T_1$ at $p_1,\ldots, p_\delta$ (and may be other singularities which we do not care about), the first order infinitesimal deformations of $S_0$ in $\mathcal X$ which are equisingular at $p_1,\ldots, p_\delta$ are a linear subspace of $H^0(S_0,\mathcal N_{ S_0|\mathcal X_0}\otimes I_{\{p_1,\dots,p_\delta\}|\mathcal X_0})$. 
\end{remark}

\begin{corollary}\label{cor:tangent-space-tacnode} Same hypotheses as in Lemma \ref {lm:upper-bound}.
Assume moreover that
\begin{equation}\label{uno}
H^1(S_0, \mathcal N_{S_0|\mathcal X_0}\otimes  I_{\{p_1,\dots,p_\delta\}|\mathcal X_0})=0
\end{equation}
or, equivalently, that
\begin{equation}\label{due}
H^1(S_0, \mathcal N_{S_0|\mathcal X_0})=0\,\,{\rm{and}}\,\,\
h^0(S_0, \mathcal N_{S_0|\mathcal X_0}\otimes  I_{\{p_1,\dots,p_\delta\}|\mathcal X_0})=h^0(S_0, \mathcal N_{S_0|\mathcal X_0})-\delta
\end{equation} 
(assuring that  the Hilbert scheme $\mathcal H^{\mathcal X_0}$ is smooth at $[S_0]$). Then
the  schemes $\mathcal {ES}_{[S_0]}^{\mathcal X_0}$ and
 $\mathbb T_\delta$  are smooth at $[S_0]$ of dimension $h^0(S_0, \mathcal N_{ S_0|\mathcal X_0}\otimes I_{\{p_1,\dots,p_\delta\}|\mathcal X_0}) )$,
 with tangent space $T_{[S_0]}(\mathcal {ES}_{[S_0]}^{\mathcal X_0})\simeq T_{[S_0]}(\mathbb T_\delta)\simeq H^0(S_0, \mathcal N_{ S_0|\mathcal X_0}\otimes I_{\{p_1,\dots,p_\delta\}|\mathcal X_0}) )\simeq H^0(S_0, \mathcal N'_{S_0|\mathcal X}). $
\end{corollary}
\begin{proof}
Consider the exact sequence
\begin{equation}\label{eq:standard-sequence1}
\xymatrix{ 0  \ar[r] &  \mathcal N_{S_0|\mathcal X_0}\otimes  I_{\{p_1,\dots,p_\delta\}|\mathcal X_0} \ar[r] &     \mathcal N_{S_0|\mathcal X_0} \ar[r] & 
\mathcal N_{S_0|\mathcal X_0}\otimes \mathcal O_{\{p_1,\dots,p_\delta\}}\ar[r] &  0, }
 \end{equation}
 from which one deduces that $h^1(S_0, \mathcal N_{ S_0|\mathcal X_0}\otimes I_{\{p_1,\dots,p_\delta\}|\mathcal X_0})=0$ if and only if $h^1(S_0, \mathcal N_{ S_0|\mathcal X_0})=0$ and
 $$
h^0(S_0,\mathcal N_{ S_0|\mathcal X_0}\otimes I_{\{p_1,\dots,p_\delta\}|\mathcal X_0})=
h^0(S_0,\mathcal N_{S_0|\mathcal X_0})-\delta.
$$
Assume that $h^1(S_0, \mathcal N_{ S_0|\mathcal X_0}\otimes I_{\{p_1,\dots,p_\delta\}|\mathcal X_0})=0$. Thus $[S_0]$ is a smooth point of $\mathcal H^{\mathcal X_0}$.
In particular there exists a unique component $\mathcal H_{S_0}^{\mathcal X_0}$ of $\mathcal H^{\mathcal X_0}$ containing $[S_0]$ and having  dimension
$h^0(S_0, \mathcal N_{ S_0|\mathcal X_0})$ at $[S_0]$. Now, by  Lemma \ref{lm:upper-bound}, one has that
$$
h^0(S_0, \mathcal N_{ S_0|\mathcal X_0})-\delta=\dim(\mathcal H_{S_0}^{\mathcal X_0})-\delta\leq\dim_{[S_0]}(\mathcal {ES}_{[S_0]}^{\mathcal X_0}))\leq\dim(T_{[S_0]}(\mathcal {ES}_{[S_0]}^{\mathcal X_0})).$$
On the other hand, by \eqref{eq:equisingular1}, one has that
{\footnotesize
$$
\dim(T_{[S_0]}(\mathcal {ES}_{[S_0]}^{\mathcal X_0}))\leq h^0(S_0,\mathcal N'_{S_0|\mathcal X})\leq  h^0(S_0,\mathcal N_{ S_0|\mathcal X_0}\otimes I_{\{p_1,\dots,p_\delta\}|\mathcal X_0})=h^0(S_0, \mathcal N_{ S_0|\mathcal X_0})-\delta.$$}
The corollary follows.
\end{proof}

We note the following:

\begin{lemma}\label{lm:speranza-fallita}
Let $[S_0]\in \mathcal H^{\mathcal X|\mathbb D}$  be any point corresponding to a reduced effective Cartier divisor $S_0=S_A\cup S_B\subset\mathcal X_0$.
Assume that $H^1(S_0,\mathcal N_{S_0|\mathcal X_0})=0$. Then the space of 
first order infinitesimal deformations of $S_0$ in $\mathcal X$ is given by
$$ H^0(S_0,\mathcal N_{S_0|\mathcal X})\simeq H^0(S_0, \mathcal N_{S_0|\mathcal X_0})\oplus H^0(S_0,\mathcal O_{S_0})$$
and
$$ H^1(S_0,\mathcal N_{S_0|\mathcal X})\simeq H^1(S_0,\mathcal O_{S_0})$$
is an obstruction space for $\mathcal O_{\mathcal H^{\mathcal X|\mathbb D}, [S_0]}.$

\end{lemma}
\begin{proof}
By the hypothesis  we have  ${\rm Ext}^1(\mathcal O_{S_0},\mathcal N_{{S_0}|\mathcal X_0})\simeq H^1(S_0,\mathcal N_{S_0|\mathcal X_0})=0$ and by
 the exact sequence 
\begin{equation}\label{eq:normali}
\xymatrix{
 0  \ar[r] &       \mathcal N_{S_0|\mathcal X_0}  \ar[r] &      \mathcal N_{S_0|\mathcal X}\ar[r] &
 \mathcal N_{\mathcal X_0|\mathcal X}|_{S_0}=\mathcal O_{S_0} \ar[r] &  0, }
 \end{equation}
we have that 
$$
\mathcal N_{S_0|\mathcal X}\simeq \mathcal N_{S_0|\mathcal X_0}\oplus \mathcal O_{S_0}.
$$
The statement then follows by standard deformation theory. \end{proof}

\begin{corollary}\label{cor:klop} Let $[S_0]\in \mathcal H^{\mathcal X|\mathbb D}$  be a point corresponding to a reduced effective Cartier divisor $S_0=S_A\cup S_B\subset\mathcal X_0$. Assume that $[S_0]$ belongs to an irreducible component $\mathcal H$ of $\mathcal H^{\mathcal X|\mathbb D}$ that dominates $\mathbb D$.
Suppose  that $H^1(S_0,\mathcal N_{S_0|\mathcal X_0})=0$. Then $[S_0]$ is a smooth point for $\mathcal H^{\mathcal X|\mathbb D}$ and $\dim(\mathcal H)=\dim(\mathcal H^{\mathcal X_0}_{[S_0]})+1$.

\end{corollary}

\begin{proof} One has $\dim(\mathcal H)\geq \dim (\mathcal H^{\mathcal X_0}_{[S_0]})+1=h^0(S_0, \mathcal N_{S_0|\mathcal X_0})+1$. On the other hand, 
$\dim(\mathcal H)\leq h^0(S_0,\mathcal N_{S_0|\mathcal X})=h^0(S_0, \mathcal N_{S_0|\mathcal X_0})+1$ by Lemma \ref {lm:speranza-fallita}. The assertion follows. \end{proof}

\begin{corollary}\label {cor:ind}
In the same setting as in Lemma \ref {lm:tangent-space-tacnode} and same hypotheses as in Lemma \ref {lm:upper-bound},  suppose that  \eqref{uno} (or equivalently \eqref{due}) holds,
assuring that $\mathbb T_\delta$ is smooth at $[S_0]$ of dimension $h^0(S_0, \mathcal N_{S_0|\mathcal X_0}\otimes  I_{\{p_1,\dots,p_\delta\}|\mathcal X_0})=h^0(S_0, \mathcal N_{S_0|\mathcal X_0})-\delta$.  Then for every positive integer $r<\delta$, the variety $\mathbb T_r$ is non--empty and $[S_0]\in \mathbb T_r$. More precisely, in an analytic neighborhood of $[S_0]$, $\mathbb T_r$ consists of $\delta \choose r$ smooth analytic branches each of dimension $h^0(S_0,\mathcal N_{ S_0|\mathcal X_0})-r$ that intersect at $[S_0]$ along a smooth analytic branch of $\mathbb T_\delta$.

Informally speaking, this is as saying that the $\delta$ singularities of type $T_1$ of $S_0$ can be {\rm independently smoothed} inside $\mathcal  X_0$. 
\end{corollary}

\begin{proof}  We first prove the assertion for $r=1$. 

 Let $\mathbb T$ be the closure of the subset of the Hilbert scheme  $\mathcal H_{S_0}^{\mathcal X_0}$
 consisting of all surfaces $S_0'$ such that the intersection curve of $S_0'$ with $R=A\cap B$ is singular with at most nodes.  Observe that  $[S_0]\in \mathbb T$. 
 
 We claim that any irreducible component of $\mathbb T$ that contains $[S_0]$ has exactly codimension 1 in  $\mathcal H_{S_0}^{\mathcal X_0}$. Indeed, let $\mathbb T'$ be such a component. Consider the dominant map $\varphi_{|\mathcal H^{\mathcal X_0}_{S_0}}$ as in \eqref {eq:restr}, which is defined at a general point of $\mathbb T'$. 
 
 By our hypotheses, the general element in $\mathcal H_C^R$ is a smooth curve, hence $\mathbb T'$ has codimension at least 1 in $\mathcal H_{S_0}^{\mathcal X_0}$. Let $\mathcal T$ be the image of the restriction of $\varphi_{|\mathcal H^{\mathcal X_0}_{S_0}}$ to $\mathbb T'$. Then $\mathcal T$ has codimension 1 in $\mathcal H_C^R$.

Let $\alpha$ be the dimension of the general fibre of $\varphi_{|\mathcal H^{\mathcal X_0}_{S_0}}$ and let $\beta$ be the dimension of the general fibre of the restriction of $\varphi_{|\mathcal H^{\mathcal X_0}_{S_0}}$ to $\mathbb T'$.  One has $\alpha\leq \beta$. Then
$$
\dim (\mathcal H_{S_0}^{\mathcal X_0})=\dim(\mathcal H_C^R)+\alpha
$$
and 
$$
\dim (\mathbb T')=\dim(\mathcal T)+\beta=\dim(\mathcal H_C^R)-1+\beta\geq \dim(\mathcal H_C^R)-1+\alpha=\dim (\mathcal H_{S_0}^{\mathcal X_0})-1.
$$
Since $\dim (\mathbb T')<\dim (\mathcal H_{S_0}^{\mathcal X_0})$ we have
$\dim (\mathbb T')=\dim (\mathcal H_{S_0}^{\mathcal X_0})-1$ and $\alpha=\beta$, as claimed.

Now, we consider a suitably small analytic open neighborhood $U$ of $[S_0]$ in $\mathbb T$. Every surface $S_0'$ such that $[S_0']\in U$ has at most $\delta$ singularities of type $T_1$. 

Consider  the variety $I\subset U\times R$ consisting of all pairs $([S_0'],q)$ with $q$ a $T_1$ singularity of $S_0'$.  Let $\pi_1: I\to U$ and $\pi_2: I\to R$ be the two projections. The former one  has finite fibres, implying the every irreducible component of $I$ has dimension 
$\dim (\mathcal H_{S_0}^{\mathcal X_0})-1$.  As for the latter,  it is dominant,  because we assume the hypotheses of Lemma \ref{lm:upper-bound}. Moreover, if $q\in R$ is a point, the fibre $\pi_2^{-1}(q)$ is the locally closed set of surfaces $S_0'$ in   $\mathcal H_{S_0}^{\mathcal X_0}$ having a $T_1$ singularity in $q$. 

Let $V_i$ be a sufficiently small analytic neighborhood of $p_i$ in $R$ for $i=1,\ldots, \delta$.  By the above considerations, 
 $\pi_1(\pi_2^{-1}(V_i))\subseteq U$ is an analytic open set that parametrizes deformations of $S_0$ which are analytically equisingular at $p_i$. 
 Hence, by Remark \ref {rem:gen}, the tangent space to $\pi_1(\pi_2^{-1}(V_i))$ at $[S_0]$ is contained in $H^0(S,\mathcal N_{ S_0|\mathcal X_0}\otimes I_{p_i|\mathcal X_0})$
and
$$
\dim_{[S_0]}(\mathcal H^{\mathcal X_0}_{S_0})-1=\dim_{[S_0]}(\pi_1(\pi_2^{-1}(V_i))) \leq 
h^0(S_0,\mathcal N_{ S_0|\mathcal X_0}\otimes I_{p_i|\mathcal X_0}).
$$
If \eqref{uno} holds, then $h^1(S_0,\mathcal N_{ S_0|\mathcal X_0}\otimes I_{p_i|\mathcal X_0})=0$ for all $i=1,\ldots, \delta$, 
and  one has 
$$h^0(S_0,\mathcal N_{ S_0|\mathcal X_0}\otimes I_{p_i|\mathcal X_0})=h^0(S_0,\mathcal N_{ S_0|\mathcal X_0})-1=\dim_{[S_0]}(\mathcal H^{\mathcal X_0}_{S_0})-1.$$
Thus $\pi_1(\pi_2^{-1}(V_i))$, that is an open analytic subset of $\mathcal {ES}_{[S_0],p_i}^{\mathcal X_0}$, is an analytic branch of $U$ of dimension $h^0(S_0,\mathcal N_{ S_0|\mathcal X_0})-1$, smooth at $[S_0]$. 

Next we prove that the general element $[S_0']$ in $\pi_1(\pi_2^{-1}(V_i))$ has a unique $T_1$ singularity. We argue for the  case $i=1$ and the proof is analogous in the other cases. Suppose this is not the case, and that $S_0'$ has $s$ singularities $q_1,\ldots, q_s$ of type $T_1$ with $s>1$.  When $S_0'$ specializes to $S_0$, $q_1,\ldots, q_s$ specialize, say, to $p_1,\ldots, p_s$. By the same argument as above, the tangent space to $\pi_1(\pi_2^{-1}(V_1))$ at $[S_0']$ is contained in $H^0(S_0' ,\mathcal N_{ S_0'|\mathcal X_0}\otimes I_{\{q_1,\ldots, q_s\}|\mathcal X_0})$ and,  under the hypothesis \eqref{uno}, one has
{\tiny
$$
h^0(S_0',\mathcal N_{ S'_0|\mathcal X_0}\otimes I_{\{q_1,\ldots, q_s\}|\mathcal X_0})\leq 
h^0(S_0,\mathcal N_{ S_0|\mathcal X_0}\otimes I_{\{p_1,\ldots, p_s\}|\mathcal X_0})=
h^0(S_0,\mathcal N_{ S_0|\mathcal X_0})-s<h^0(S_0,\mathcal N_{ S_0|\mathcal X_0})-1
$$}
and this is a contradiction. This proves the assertion for $r=1$. 

Consider now the case $\delta>r>1$. Fix $p_{i_1},\ldots, p_{i_r}$ distinct points among $p_1,\ldots, p_\delta$. The intersection 

$$ \mathfrak T_{i_1,\ldots, i_r}:=\bigcap_{j=1}^r \pi_1(\pi_2^{-1}(V_{i_j})),$$
that is an  analytic open subset of $\mathcal{ES}_{[S_0],p_{i_1},\ldots, p_{i_r}}^{\mathcal X_0}$, 
is an analytic variety  in  $\mathcal H^{\mathcal X_0}_{S_0}$
parametrizing deformations of $S_0$ that are analytic equisingular at the points $p_{i_1},\ldots, p_{i_r}$. With the same argument as above, one sees that,  under the hypothesis \eqref{uno}, 
$\mathfrak T_{i_1,\ldots, i_r}$ is smooth of codimension $r$ in  $\mathcal H_{S_0}^{\mathcal X_0}$, with tangent space at $[S_0]$ given by $H^0(S_0,\mathcal N_{ S_0|\mathcal X_0}\otimes I_{\{p_{i_1},\dots,p_{i_r}\}|\mathcal X_0})$.

Moreover, again by the same argument as above, the general  element $S'_0$ in $\mathfrak T_{i_1,\ldots, i_r}$ has exactly $r$ singularities of type $T_1$ at points specializing to $p_{i_1}, \ldots, p_{i_r}$ when $S'_0$ specializes to $S_0$. So $\mathfrak T_{i_1,\ldots, i_r}$ is a smooth analytic branch of $\mathbb T_r$ containing $[S_0]$, and 
this ends the proof of the corollary. \end{proof}

Let now $\mathbb T_{\delta_A,\delta_B,\delta_R}\subseteq \mathcal H^{\mathcal X_0}$ be the Zariski closure of the family of surfaces $S_0=S_A\cup S_B$ in $\mathcal X_0$ with $\delta_A$ nodes on $A$ and $\delta_B$ nodes on $B$ off $R$ and $\delta_R$ singularities of type $T_1$ on $R$.

\begin{corollary}\label{cor:defo}
Let $S_0=S_A\cup S_B$ be a reduced effective Cartier divisor such that $S_A$ and $S_B$ have respectively $\delta_A$ and $\delta_B$ nodes $p_{A,1},\ldots, p_{A,\delta_A}$ and $p_{B,1},\ldots, p_{B,\delta_B}$ off $R$,  are elsewhere smooth and  intersect transversally along a curve $C=S_A\cap S_B$, except
for $\delta_R$ distinct points $p_{R,1}, \ldots, p_{R,\delta_R} \in C\subset R$ where $S_0$ has singularities of type $T_1$. Let  $\mathcal {ES}^{\mathcal X_0}_{[S_0]}$ be the locally closed  set of equisingular deformations of $S_0$ in $\mathcal X_0$. Consider the ideal sheaf
$I_{\mathfrak Z|\mathcal X_0}$ in $\mathcal X_0$ of the $0$--dimensional reduced scheme  $\mathfrak Z$  of lenght $\delta=\delta_A+\delta_B+\delta_R$ given by
$$
\mathfrak Z= \sum_{i=1}^{\delta_A} p_{A,i}+\sum_{i=1}^{\delta_B} p_{B,i}+\sum_{i=1}^{\delta_R} p_{R,i}.
$$
Then 
\begin{equation}\label{tangente-nodo}
T_{[S_0]}(\mathcal {ES}^{\mathcal X_0}_{[S_0]})\subseteq H^0(S_0, \mathcal N_{S_0|\mathcal X_0}\otimes  \mathcal I_{\mathfrak Z|\mathcal X_0}).
\end{equation}

If
\begin{equation}\label{uno-nodi}
h^1(S_0, \mathcal N_{S_0|\mathcal X_0}\otimes  \mathcal I_{\mathfrak Z|\mathcal X_0})=h^1(C, \mathcal N_{C|R}\otimes  \mathcal I_{\{p_{R,1}, \ldots, p_{R,\delta_R} \}|R})=0
\end{equation}
and the map
 $$
 \varphi_{|\mathcal H^{\mathcal X_0}_{S_0}}: \mathcal H^{\mathcal X_0}_{S_0}\dasharrow \mathcal H_C^R
 $$ 
 defined as in Lemma \ref {lm:upper-bound} is dominant,
then the equality holds in \eqref{tangente-nodo}, and the locally closed set $\mathcal {ES}^{\mathcal X_0}_{[S_0]}$ of locally trivial deformations of $S_0$ in $\mathcal X_0$ is smooth at $[S_0]$ of dimension $h^0(S_0, \mathcal N_{S_0|\mathcal X_0}\otimes  I_{\mathfrak Z|\mathcal X_0})=h^0(S_0, \mathcal N_{S_0|\mathcal X_0})-\delta$. In particular there exists only one irreducible component $\mathbb T\subset \mathbb T_{\delta_A,\delta_B,\delta_R}$
containing the point $[S_0]$ (which is smooth at $[S_0]$ and contains $\mathcal {ES}^{\mathcal X_0}_{[S_0]}$ as a Zariski open set). 
Moreover, under these hypotheses, the singularities of $S_0$ may be smoothed independently in $\mathcal X_0$. More precisely,
for every $\delta_A'\leq \delta_A$, $\delta_B'\leq\delta_B$ and $\delta'_R\leq\delta_R$ we have that
 $\mathbb T_{\delta_A',\delta_B',\delta_R'}$ is non-empty and $[S_0]\in\mathbb T_{\delta_A',\delta_B',\delta_R'}$. In an analytic neighborhood of $[S_0]$, 
 $\mathbb T_{\delta_A',\delta_B',\delta_R'}$ consists of
$$
{\delta_R \choose \delta'_R}\, {\delta_A \choose \delta'_A}\, {\delta_B \choose \delta'_B}
$$
smooth analytic branches of dimension $h^0(S_0, \mathcal N_{S_0|\mathcal X_0})-\delta'$, where 
$\delta'=\delta_A'+\delta_B'+\delta_R'$, that intersect at $[S_0]$ along a smooth analytic branch of $\mathbb T\subset\mathbb T_{\delta_A,\delta_B,\delta_R}$, corresponding to deformations of $[S_0]$ preserving $\delta_R'$ points of type $T_1$ and  $\delta_A'$ nodes on $A$ and $\delta_B'$ nodes on $B$.
\end{corollary}

\begin{proof}
Let $S_0$ be a surface as in the statement. The inclusion \eqref{tangente-nodo} follows from Lemma \ref{lm:tangent-space-tacnode}, Remark \ref{rem:gen} and  well known deformation theory of nodal surfaces (see \cite [\S 2.3] {GK}). It can be proved by using \eqref{eq:standard-sequence}. In particular,
if one localizes \eqref{eq:standard-sequence} at a node $p$ of $S_0$, then $H^0(S_0,T^1_{{S_0},p})\cong \mathbb C$ can be identified with the tangent space to the versal deformation space of a node. 

Now we want to prove that, under the  hypotheses of the corollary, the locally closed set $\mathcal {ES}^{\mathcal X_0}_{[S_0]}$ is smooth at $[S_0]$ of 
 codimension $\delta$ in the Hilbert scheme $\mathcal H^{\mathcal X_0}_{S_0}$.  

Let $C\subset R$ be the $\delta_R$-nodal curve cut out by $S_0$ on $R$. By the hypothesis \eqref{uno-nodi}, one has that 
$h^1(C, \mathcal N_{C|R}\otimes \mathcal I_{\{p_{R,1}, \ldots, p_{R,\delta_R}\}|R})=0.$ This implies that $[C]$ is a smooth point
of the locally closed Severi variety $\mathcal V_\delta$ of $\delta$-nodal curves in $\mathcal H^R_C$. Let $V$ be the unique irreducible component of $\mathcal V_\delta$
containing $[C]$. Now, as we saw in Lemma \ref{lm:upper-bound}, $ \varphi_{|\mathcal H^{\mathcal X_0}_{S_0}}^{-1}(V)$
has at least one irreducible component of codimension at most $\delta_R$ in  $\mathcal H^{\mathcal X_0}_{S_0}$. On the other hand, by Lemma \ref{lm:tangent-space-tacnode} we have that
$$T_{[{S_0}]}( \varphi_{|\mathcal H^{\mathcal X_0}_{S_0}}^{-1}(V))\subseteq H^0(S_0, \mathcal N_{S_0|\mathcal X_0}\otimes \mathcal I_{\{p_{R,1}, \ldots, p_{R,\delta_R}\}|\mathcal X_0})$$
and   by \eqref{uno-nodi} we have  $h^0(S_0, \mathcal N_{S_0|\mathcal X_0}\otimes \mathcal I_{\{p_{R,1}, \ldots, p_{R,\delta_R}\}|\mathcal X_0})=\dim(\mathcal H^{\mathcal X_0}_{S_0})-\delta_R$.
Thus $\varphi_{|\mathcal H^{\mathcal X_0}_{S_0}}^{-1}(V)$ is smooth at $[S_0]$ of codimension $\delta_R$. 
 We observe that $\varphi_{|\mathcal H^{\mathcal X_0}_{S_0}}^{-1}(V)$  is an analytic open set of the variety
 $\mathcal {ES}^{\mathcal X_0}_{[S_0],p_{R,1}, \ldots, p_{R,\delta_R}}$ of deformations of $S_0$ in $\mathcal X_0$ that are locally trivial at every $T_1$ singularity $p_{R,i}$ and
 we just proved that
 
 $$
 T_{[S_0]}(\mathcal {ES}^{\mathcal X_0}_{[S_0],p_{R,1}, \ldots, p_{R,\delta_R}})=T_{[S_0]}( \varphi_{|\mathcal H^{\mathcal X_0}_{S_0}}^{-1}(V))=
  H^0(S_0, \mathcal N_{S_0|\mathcal X_0}\otimes \mathcal I_{\{p_{R,1}, \ldots, p_{R,\delta_R}\}|\mathcal X_0}).$$

We morever observe that the general element
$[S'_0]$ of $\varphi_{|\mathcal H^{\mathcal X_0}_{S_0}}^{-1}(V)$ corresponds to a surface $S'_0=S'_A\cup S'_B$, where $S'_A$ and $S'_B$ intersect transversally along a curve $C'$ on $R$, except for $\delta_R$ points $p_{R,1}', \ldots, p_{R,\delta_R}'\in C'$, which are singularities of type $T_1$
of $S'_0$, and specialize to $p_{R,1}, \ldots, p_{R,\delta_R}$ as $S_0'$ specializies to $S_0$. 

We claim that $S_A'$ and $S_B'$ are smooth outside $R$.
 Indeed, since $[S_0]$ belongs to $ \varphi_{|\mathcal H^{\mathcal X_0}_{S_0}}^{-1}(V)$, the surface $S'_0$ may have at most $\delta'_A\leq\delta_A$ nodes $p_{A,1},\ldots, p_{A,\delta'_A}$   on $A$ and $\delta'_B\leq\delta_B$
nodes $p_{B,1},\ldots, p_{B,\delta'_B}$ on $B$, deformations of $\delta_A'$   nodes of $S_0$ on $A$ and $\delta_B'$ nodes of $S_0$ on $B$.
If this happens, denoting by $\mathfrak Z'$ the scheme of singular points of $S'_0$, then 
$
 T_{[S'_0]}(\mathcal {ES}^{\mathcal X_0}_{[S_0],p_{R,1}, \ldots, p_{R,\delta_R}})=T_{[S'_0]}\varphi_{|\mathcal H^{\mathcal X_0}_{S_0}}^{-1}(V)\subseteq
  H^0(S'_0,\mathcal N_{S'_0|\mathcal X_0}\otimes I_{\mathfrak Z'|\mathcal X_0}).
$
But, once again by \eqref{uno-nodi} and by semicontinuity, one has that $h^0(S_0',\mathcal N_{S_0'|\mathcal X_0}\otimes I_{\mathfrak Z'|\mathcal X_0})=
H^0(S'_0,\mathcal N_{S'_0|\mathcal X_0})-\delta_R-\delta'_A-\delta_B'.$ It follows that $\delta'_A=\delta_B'=0$, i.e. $S'_A$ and $S'_B$ are smooth off $R$
and  $\varphi_{|\mathcal H^{\mathcal X_0}_{S_0}}^{-1}(V)$  is a locally closed set in one irreducible component $\mathbb T\subset\mathbb T_{\delta_R}$,
smooth at $[S_0]$.
 We just proved that, under our hypotheses,  one may deform $S_0$ in $\mathcal X_0$ by smoothing all nodes of $S_0$ and  by preserving all $T_1$ singularities.

Let now $\mathfrak Z_A$ and $\mathfrak Z_B$ be respectively the scheme of nodes of $S_0$ on $A$ and $B$. Let $\mathcal {ES}^{\mathcal X_0}_{[S_0],\mathfrak Z_A,\mathfrak Z_B}$ the  scheme of deformations of $S_0$ which are locally trivial at every node of $S_0$. By standard deformation theory of nodal surfaces, one has that, under the hypothesis $h^1(S_0, \mathcal N_{S_0|\mathcal X_0}\otimes  \mathcal I_{\mathfrak Z_A\cup \mathfrak Z_B|\mathcal X_0})=0$ (that holds by \eqref{uno-nodi}), $\mathcal {ES}^{\mathcal X_0}_{[S_0],\mathfrak Z_A,\mathfrak Z_B}$ is smooth of codimension $\delta_A+\delta_B$ in $\mathcal H^{\mathcal X_0}_{S_0}$ at $[S_0]$ and moreover 
$$T_{[S_0]}(\mathcal {ES}^{\mathcal X_0}_{[S_0],\mathfrak Z_A,\mathfrak Z_B})=H^0(S_0, \mathcal N_{S_0|\mathcal X_0}\otimes  \mathcal I_{\mathfrak Z_A\cup \mathfrak Z_B|\mathcal X_0}).$$
 With a similar  argument as above, one sees that the general element $[\tilde S_0]$ of $\mathcal {ES}^{\mathcal X_0}_{[S_0],\mathfrak Z_A,\mathfrak Z_B}$ corresponds to a surface $\tilde S_0=\tilde S_A\cup\tilde S_B$, with 
$\tilde S_A$ and $\tilde S_B$ intersecting transversally along a smooth curve $\tilde C\subset R$ and having, respectively, $\delta_A$ and $\delta_B$ nodes as singularities.
In particular $\mathcal {ES}^{\mathcal X_0}_{[S_0],\mathfrak Z_A,\mathfrak Z_B}$ is a locally closed set in an irreducible component $\tilde{\mathbb T}$
of $\mathbb T_{\delta_A,\delta_B,0}$, of which $[S_0]$ is a smooth point.

Now the equisingular deformation locus $\mathcal {ES}^{\mathcal X_0}_{[S_0]}$ of $S_0$ in $\mathcal X_0$
is the intersection of the loci $ \mathcal {ES}^{\mathcal X_0}_{[S_0],p_{R,1}, \ldots, p_{R,\delta_R}}$ and $ \mathcal {ES}^{\mathcal X_0}_{[S_0],\mathfrak Z_A,\mathfrak Z_B}$. Hence $\mathcal {ES}^{\mathcal X_0}_{[S_0]}$ has codimension at most $\delta$ in $\mathcal H^{\mathcal X_0}_{S_0}$, because $[S_0]$ is a smooth point of $\mathcal H^{\mathcal X_0}_{S_0}$.

On the other hand, one has

\begin{eqnarray*}
T_{[S_0]}(\mathcal {ES}^{\mathcal X_0}_{[S_0]}) & = & T_{[S_0]}(\mathcal {ES}^{\mathcal X_0}_{[S_0],p_{R,1}, \ldots, p_{R,\delta_R}})\cap 
T_{[S_0]}(\mathcal {ES}^{\mathcal X_0}_{[S_0],\mathfrak Z_A,\mathfrak Z_B})\\
& = & H^0(S_0, \mathcal N_{S_0|\mathcal X_0}\otimes \mathcal I_{\{p_{R,1}, \ldots, p_{R,\delta_R}\}|\mathcal X_0})
\cap H^0(S_0, \mathcal N_{S_0|\mathcal X_0}\otimes  \mathcal I_{\mathfrak Z_A\cup \mathfrak Z_B|\mathcal X_0})\\
& = & H^0(S_0, \mathcal N_{S_0|\mathcal X_0}\otimes  \mathcal I_{\mathfrak Z|\mathcal X_0}).
\end{eqnarray*}

By \eqref{uno-nodi}, we have 
$$
h^0(S_0, \mathcal N_{S_0|\mathcal X_0}\otimes  \mathcal I_{\mathfrak Z|\mathcal X_0})=h^0(S_0, \mathcal N_{S_0|\mathcal X_0})-\delta
$$
this proves that $\mathcal {ES}^{\mathcal X_0}_{[S_0]}$ is smooth at $[S_0]$ of codimension exactly $\delta$ in $\mathcal H^{\mathcal X_0}_{S_0}$, as wanted. This proves the first part of the corollary. 

The second part is proved with analogous arguments as the ones used in the proof of Corollary \ref {cor:ind}. \end{proof}

\begin{proposition}\label{prop:defonodi} Let $S_0=S_A\cup S_B$ be a reduced effective Cartier divisor as in the statement of Corollary \ref {cor:defo}. 
Assume that $[S_0]$ belongs to an irreducible component $\mathcal H$ of $\mathcal H^{\mathcal X|\mathbb D}$ that dominates $\mathbb D$ and that \eqref {uno-nodi} holds.

Let $\mathcal {ES}^\mathcal X_{[S_0],\mathfrak Z_A,\mathfrak Z_B}$ be the locus in $\mathcal H^\mathcal X$ of deformations of $S_0$ which are equisingular at every node of $S_0$. Then $\mathcal {ES}^\mathcal X_{[S_0],\mathfrak Z_A,\mathfrak Z_B}$ is generically smooth of codimension $\delta_A+\delta_B$ in $\mathcal H$ and it contains $\mathcal {ES}^{\mathcal X_0}_{[S_0],\mathfrak Z_A,\mathfrak Z_B}$ as a subscheme of codimension 1. 

In simple words, $[S_0]$ can be deformed out of $\mathcal X_0$ preserving the $\delta_A+\delta_B$ nodes. 

\end{proposition} 

\begin{proof} From the hypotheses, it follows that $H^1(S_0,\mathcal N_{S_0|\mathcal X_0})=0$. Then, by Corollary \ref {cor:klop}, $\mathcal H^{\mathcal X|\mathbb D}$  is smooth at $[S_0]$ with dimension  $h^0(S_0,\mathcal N_{S_0|\mathcal X_0})+1=\dim(\mathcal H_{S_0}^{\mathcal X_0})+1$. By standard deformation theory, there is an analytic neighborhood $U$ of $[S_0]$ in $\mathcal H$ and a versal morphism
$$
f: U\longrightarrow \prod_{p\in \mathfrak Z_A+\mathfrak Z_B} \Delta_p
$$
where $\Delta_p$ is the versal deformation space of a node, and therefore it has dimension 1. Let $U'=U\cap \mathcal H_{S_0}^{\mathcal X_0}$. Then $f$ restricts to
$$
g: U'\longrightarrow \prod_{p\in \mathfrak Z_A+\mathfrak Z_B} \Delta_p
$$
The  differential of $g$ at $[S_0]$ is
$$
H^0(S_0, N_{S_0|\mathcal X_0})\to \prod_{p\in  \mathfrak Z_A+\mathfrak Z_B}  T_{S_0,p}^1\cong \mathbb C^{\delta_A+\delta_B}
$$
and this map is surjective since its kernel is $H^0(S_0, \mathcal N_{S_0|\mathcal X_0}\otimes  \mathcal I_{\mathfrak Z_A\cup \mathfrak Z_B|\mathcal X_0})$, which has codimension $\delta_A+\delta_B$ in $H^0(S_0, N_{S_0|\mathcal X_0})$ by \eqref{uno-nodi}. Hence $g$ has maximal rank at $[S_0]$ and therefore also $f$ is of maximal rank at $[S_0]$. Hence $f^{-1}(0)$ and $g^{-1}(0)$ are  analytic subvarieties of $U$ and $U'$ respectively, smooth at $[S_0]$ and of codimension $\delta_A+\delta_B$ in $U$ and $U'$ respectively. By versality, $f^{-1}(0)$ (resp. $g^{-1}(0)$) coincides with $\mathcal {ES}^{\mathcal X}_{[S_0],\mathfrak Z_A,\mathfrak Z_B}$ (resp. $\mathcal {ES}^{\mathcal X_0}_{[S_0],\mathfrak Z_A,\mathfrak Z_B}$). The statement follows. \end{proof}

\subsubsection{Global deformations of surfaces with $T_1$ singularities to nodal surfaces}\label{sect:smoothing-to-nodes}

In this section we will assume the following set up. We have the family $\pi: \mathcal X\to \mathbb D$ as usual 
with its relative Hilbert scheme $\mathcal H^{\mathcal X|\mathbb D}$, whose fibre over $t\in\mathbb D$ is the Hilbert scheme 
of $\mathcal H^{\mathcal X_t}$ of $\mathcal X_t$.

Let $\mathcal V_\delta^{\mathcal X|\mathbb D}$ be the Zariski closure  in $\mathcal H^{\mathcal X|\mathbb D}$ 
of the relative Severi variety 
$\mathcal W^{\mathcal X\setminus\mathcal X_0|\mathbb D\setminus 0}_\delta\subset\mathcal H^{\mathcal X\setminus\mathcal X_0|\mathbb D\setminus 0}$
of $\delta$-nodal surfaces. 
 We want to provide sufficient conditions for  $\mathcal V_\delta^{\mathcal X|\mathbb D}$ to be non-empty.


We will  suppose that we have a line bundle $\mathcal L$ on $\mathcal X$ with the following properties:\\
\begin{inparaenum}
 \item [(1)]  $h^0(\mathcal X_t, \mathcal L_t)$ is a constant $r+1$ in $t$ and greater or equal than 4. In particular every surface $S_0$ in $ |\mathcal L_0|$ belongs to an irreducible component $\mathcal H$ of the relative Hilbert scheme $\mathcal H^{\mathcal X|\mathbb D}$ that dominates $\mathbb D$;\\
 \item [(2)] $|\mathcal L_0|$ is base point free, so that we can assume that $|\mathcal L_t|$ is base point free for all $t\in \mathbb D$;\\  
\item [(3)]  if $p_t\in \mathcal X_t$ is a general point, then the general surface in $|\mathcal L_t|$ with a singular point at $p_t$ is singular only at finitely many points, for the general $t\in \mathbb D$.
\end{inparaenum}

In this setting we can consider the rank $r$ projective bundle $\bar \pi: \mathbb P(\pi_*(\mathcal L))\to \mathbb D$. A point in $\mathbb P:=\mathbb P(\pi_*(\mathcal L))$ that maps to $t\in \mathbb D$ is a non--zero section of $H^0(\mathcal X_t, \mathcal L_t)$ up to a constant. In particular, if $t\neq 0$, a point in $\mathbb P$ corresponds to a surface in $|\mathcal L_t|$. Consider the open Zariski subset $\mathbb P':=\bar \pi^{-1}(\mathbb D\setminus \{0\})$, which, by the above considerations, can be regarded as a subvariety of the relative Hilbert scheme of surfaces in $\mathcal X$.  By a standard parameter count, one sees that there is a subscheme $Z$ of pure codimension 1 in $\mathbb P'$ whose points correspond to sections vanishing along singular surfaces. We will denote by $\bar Z$ the closure of $Z$ in $\mathbb P$, that has also codimension 1.

\begin{proposition}\label{lm:delta=1} Set up as above with the following further condition:  the subspace of sections of $H^0(\mathcal X_0,\mathcal L_0)$ that vanish on $R=A\cap B$, with $A$ and $B$ the irreducible components of $\mathcal X_0$, has codimension strictly larger than 1 in $H^0(\mathcal X_0,\mathcal L_0)$. 
 
Let  $S_0=S_A\cup S_B\subset \mathcal X_0$  be a surface corresponding to a section of $\mathcal L_0$. We suppose that:\\
\begin{inparaenum} 
\item [(a)] 
 $S_A$ and $S_B$ are smooth and  intersect transversally along a curve $C=S_A\cap S_B$, except
for a point $p=p_1\in C\subset R$ where $S_0$ has a singularity of type $T_1$ and the hypotheses of Lemma \ref {lm:upper-bound} hold for $\delta=1$; \\
\item [(b)] the sublinear system $\mathcal L_0(2,p)$ of $|\mathcal L_0|$ of surfaces with at least a $T_1$ singularity at $p$ has codimension 3 in $|\mathcal L_0|$;\\
\item [(c)]  \eqref{uno} (for $\delta=1$) holds for $S_0$. 
\end{inparaenum}
 
Then:\\
\begin{inparaenum}
\item [(i)]  $\mathbb T_1$ is smooth at $[S_0]$ of codimension $1$ in $\mathcal H_0$;\\ 
\item [(ii)] $S_0$ can be deformed to a $1$-nodal surface $S_t\subset\mathcal X_t$;\\
\item [(iii)] if $\mathbb T\subseteq \mathbb T_1$ is the unique irreducible component containing $[S_0]$, then there exists a reduced, irreducible component $\mathcal V\subset\mathcal V^{\mathcal X|\mathbb D}_1$ of dimension $\dim(\mathcal H)-1$ whose central fibre $\mathcal V_0$ contains $\mathbb T$ as an irreducible component.\\
\end{inparaenum}
\end{proposition}

Before giving the proof of the proposition, we make a preliminary lemma. For this we need some notation. Let $I\subset |\mathcal L_0|\times R$, with $R=A\cap B$, be the locally closed subset consisting of pairs $(S_0,p)$ such that $S_0$ cuts out on $R$ a curve singular at $p$. We will consider the two projections $\pi_1: I\to |\mathcal L_0|$ and $\pi_2: I\to R$. Note that if $p\in R$, then $\pi_2^{-1}(p)$ can be identified with $\mathcal L_0(2,p)$. 

\begin{lemma}\label{lem:dom} (i) There is at most one irreducible component $I'$ of $I$ such that the restriction of $\pi_2$ to $I'$ is dominant to $R$ via $\pi_2$.

(ii)  If $I'$ exists, and if its general element $(S_0,p)$ is such that $S_0$ cuts out on $R$ a curve with finitely many singular points,  then $\mathcal L_0(2,p)$ has codimension 3 in $|\mathcal L_0|$. Moreover $\dim(I')= \dim(|\mathcal L_0|)-1$, and its image in $|\mathcal L_0|$ via $\pi_1$ has codimension 1 in $|\mathcal L_0|$.

(iii) If there is a pair $(S_0,p)$ in $I$ such that $S_0$ cuts out on $R$ a curve with finitely many singular points and if $ \mathcal L_0(2,p)$ has codimension 3 in $|\mathcal L_0|$, then $(S_0,p)$ belongs to an irreducible component $I'$ of $I$ dominating $R$ via $\pi_2$. For this component one has $\dim(I')= \dim(|\mathcal L_0|)-1$, and its image in $|\mathcal L_0|$ via $\pi_1$ has codimension 1 in $|\mathcal L_0|$. 
\end{lemma}

\begin{proof} (i) Let $I'$ be an irreducible component of $I$ such that $I'$ is dominant to $R$ via $\pi_2$. If $p\in R$ is a general point we know that $\pi_2^{-1}(p)$ can be identified with $\mathcal L_0(2,p)$, and  $\mathcal L_0(2,p)$ is a projective space with dimension $s$ 
 independent on the general point $p$. This clearly implies that $I'$ is unique. 

(ii) Suppose the dominating component $I'$ exists. With the same notation as above we have  $\dim(I')=\dim(R)+s=s+2$.  On the other hand,
 by Remark \ref{condizioni-tacnodo},   $s\geq \dim(|\mathcal L_0|)-3$,  hence
$\dim(I')\geq \dim(|\mathcal L_0|)-1$. By the hypotheses, the map $\pi_1$, restricted to $I'$,  is generically finite onto the image and this image cannot be dense in $|\mathcal L_0|$ by Bertini's theorem. Hence $\dim(I')\leq \dim(|\mathcal L_0|)-1$, so the equality holds, and this implies that  $s=\dim(|\mathcal L_0|)-3$, as wanted.

(iii) Keep the same notation as above. The dimension of the fibre of $\pi_2$ over a general point of $R$ is $r\geq \dim(|\mathcal L_0|)-3= \mathcal L_0(2,p)\geq 0$.
Moreover there is an open dense subset $U$ of $R$, containing $p$,  such that for all $q\in U$ one has that $\mathcal L_0(2,p)$ has dimension 
$ s  = \dim(|\mathcal L_0|)-3$. Hence there is a component $I'$ of $I$ dominating $R$ via $\pi_2$.   \end{proof}

We can now give the:

\begin{proof}[Proof of Proposition \ref {lm:delta=1}]  We notice that by Corollary \ref {cor:klop}, $[S_0]$ is a smooth point for $\mathcal H^{\mathcal X|\mathbb D}$.
Part (i) follows by Corollary \ref {cor:tangent-space-tacnode}.

Let us prove part (ii). For this we go back to the notation introduced before the statement of Proposition \ref {lm:delta=1}. Consider then the intersection $\bar Z_0$ of $\bar Z$ with $\bar \pi^{-1}(0)\cong |\mathcal L_0|$, such that any of its irreducible components has codimension 1 in $|\mathcal L_0|$. By the hypotheses we made, if $\bar Z'_0$ is any irreducible component of $\bar Z_0$, its general element does not contain $R$, hence it is a surface $S_0'\in |\mathcal L_0|$ that intersects $R$ along a curve $C'$. By Proposition \ref {lm:transverse}, the curve $C'$ is singular. 

\begin{claim}\label{cl:crux} There is an irreducible component $\bar Z'_0$ of $\bar Z_0$, such that for $S_0'\in \bar Z'_0$ general, $S_0'$  intersects $R$ in a curve $C'$ that is singular at a general point $p'$ of $R$.  Moreover, $S_0'$ is limit of reduced singular surfaces $S_t\in|\mathcal L_t|$.
\end{claim} 

\begin{proof} [Proof of the Claim \ref {cl:crux}] This will be a consequence of the following  fact that we are going to prove: given a general point $p'\in R$, there is some $S_0'\in \bar Z_0$ such that the curve $C'$ cut out by $S_0'$ on $R$ is singular at $p'$. Indeed, given $p'\in R$ general, take a smooth bisection $\gamma'$ of $\mathcal X\to \mathbb D$, that passes through $p'$. As in \S \ref {ssec:surf}, we can consider the family $\mathcal Y\to \mathbb D$ obtained by desingularising the variety $\mathcal X'\to \mathbb D$ gotten via 2--fold base change $\nu_2: \mathbb D\to \mathbb D$. The variety $\mathcal Y\to \mathbb D$ has a section $\gamma$ that is mapped to $\gamma'$ via the map $\pi: \mathcal Y\to \mathcal X$. We consider $\pi^*(\mathcal L)$. Our assumption (3) implies that there are non--zero sections of $\pi^*(\mathcal L)$, on $\mathcal Y\setminus \mathcal Y_0$, vanishing with multiplicity at least 2 along $\gamma$. The assertion is now a consequence of Theorem \ref {thm-node}. \end{proof}

By the hypothesis (b) and by Lemma \ref {lem:dom}(iii), the pair $(S_0,p)$ belongs to the unique irreducible component $I'$ of $I$ dominating $R$ via $\pi_2$,  and $I'$ has dimension equal to $\dim(|\mathcal L_0|)-1$. Consider now the subset $I''\subseteq I$ of the pairs $(S_0',p')$ with $S_0'\in  \bar Z'_0$, where $\bar Z'_0$ is as in Claim \ref {cl:crux}. We notice that $I''$ also dominates $R$ via $\pi_2$. So by Lemma \ref {lem:dom}(i), $I''$ coincides with $I'$. This implies that $\bar Z'_0=\pi_1(I')$ hence $S_0\in \bar Z'_0$ and therefore the general surface $S'_0\in \bar Z'_0$ has a unique $T_1$ singularity. By Lemma \ref {lm:solo-nodo} the assertion (ii) follows. 

To prove (iii) we remark first of all that (ii) implies that $\mathcal V^{\mathcal X|\mathbb D}_1$ is non--empty and there is an irreducible component $\mathcal V$ of $\mathcal V^{\mathcal X|\mathbb D}_1$ that dominates $\mathbb D$ and contains $[S_0]$. The general point in  $\mathcal V$ corresponds to a surface $S_t$ with $t\neq 0$, with a unique node at a  general  point $p_t\in \mathcal X_t$. Moreover,  since $[S_0]$ is a smooth point of $\mathcal H$, we have that  $[S_t]$ is a smooth point of $\mathcal H$, and by the hypothesis (c) and by semicontinuity, we have that $h^1(S_t, N_{S_t|\mathcal X_t}\otimes I_{p_t})=0$. This yield that $\mathcal V\cap \mathcal H_t$ is smooth of dimension $\dim(\mathcal H_t)-1$. Hence $\mathcal V$ has dimension $\dim(\mathcal H)-1$. To prove that $\mathcal V$ is reduced, it suffices to prove that $\mathcal V$ is smooth at $[S_t]$. To see this, consider the exact sequence
$$
0\to N'_{S_t|\mathcal X}\to N_{S_t|\mathcal X}\to T^1_{S_t}\to 0
$$
where $T^1_{S_t}$ is supported on $p_t$ with stalk  $\mathbb C$, and $H^0(S_t,N'_{S_t|\mathcal X})$ is the Zariski tangent space to $\mathcal V$  at $[S_t]$. The map 
$$
H^0(S_t,N_{S_t|\mathcal X})\to T^1_{S_t}=\mathbb C
$$
is surjective because $S_t$ is smoothable inside $\mathcal H$,  by the hypothesis (2) at the beginning of this section.  Hence
$h^0(S_t,N'_{S_t|\mathcal X})=h^0(S_t,N_{S_t|\mathcal X})-1=\dim(\mathcal H)-1$, as wanted. \end{proof}

We can now prove the main result of this section extending Proposition \ref {lm:delta=1} to the case $\delta>1$:

\begin{theorem}\label{thm:main-theorem} Set up as in Proposition \ref  {lm:delta=1}. In particular we have the following  condition: the subspace of sections of $H^0(\mathcal X_0,\mathcal L_0)$ that vanish on $R=A\cap B$ with $A$ and $B$ the irreducible components of $\mathcal X_0$, has codimension strictly larger than 1 in $H^0(\mathcal X_0,\mathcal L_0)$. 
 
Let  $S_0=S_A\cup S_B\subset \mathcal X_0$  be a surface corresponding to a section of $\mathcal L_0$. We suppose that:\\
\begin{inparaenum} 
\item [(a)] 
 $S_A$ and $S_B$ have respectively $\delta_A$ and $\delta_B$ nodes $p_{A,1},\ldots, p_{A,\delta_A}$ and $p_{B,1},\ldots, p_{B,\delta_B}$ off $R$,  are elsewhere
 smooth and  intersect transversally along a curve $C=S_A\cap S_B$, except
for $\delta_R$ distinct points $p_{R,1}, \ldots, p_{R,\delta_R} \in C\subset R$ where $S_0$ has singularities of type $T_1$ and that the hypotheses of Lemma \ref {lm:upper-bound} hold;\\ 
\item [(b)] the sublinear system $\mathcal L_0(2,p_i)$ of $|\mathcal L_0|$ of surfaces with at least a $T_1$ singularity at $p_i$ has codimension 3 in $|\mathcal L_0|$, for every $1\leq i\leq \delta$;\\
\item [(c)] if $\mathfrak Z$ is the $0$--dimensional scheme of length $\delta=\delta_A+\delta_B+\delta_R$ given by
$$
\mathfrak Z= \sum_{i=1}^{\delta_A} p_{A,i}+\sum_{i=1}^{\delta_B} p_{B,i}+\sum_{i=1}^{\delta_R} p_{R,i}
$$
then $H^1(S_0, \mathcal N_{S_0|\mathcal X_0}\otimes  I_{\mathfrak Z|\mathcal X_0})=0$ (where $I_{\mathfrak Z|\mathcal X_0}$ is the ideal sheaf of the scheme $\mathfrak Z$ in $\mathcal X_0$).
\end{inparaenum}

Then:\\
\begin{inparaenum}
\item [(i)] $S$ can be deformed to a $\delta$-nodal surface $S_t\subset\mathcal X_t$;\\
\item [(ii)] if $\mathbb T\subseteq \mathbb T_{\delta_A,\delta_B,\delta_R}$ is the unique irreducible component containing $[S]$, then there exists an irreducible component $\mathcal V\subset\mathcal V^{\mathcal X|\mathbb D}_\delta$ of dimension $\dim(\mathcal H)-\delta$ whose central fibre $\mathcal V_0$ contains $\mathbb T$ as an irreducible component.\\
\end{inparaenum}
 \end{theorem}

\begin{proof}  Again, by Corollary \ref {cor:klop}, $[S_0]$ is a smooth point for $\mathcal H^{\mathcal X|\mathbb D}$. 
 Moreover, by Corollary \ref{cor:defo}, $\mathbb T_{\delta_A,\delta_B,\delta_R}$ is smooth at $[S_0]$. 

We denote by $\mathbb T$ the unique irreducible component  of $\mathbb T_{\delta_A,\delta_B,\delta_R}$  containing $[S_0]$, that is smooth at $[S_0]$. 

 Again  by Corollary \ref{cor:defo}, in an analytic neighborhood of $[S]$, 
 $\mathbb T$ consists of an analytic branch $\mathcal T$ that is the transverse intersection of $\delta$ smooth analytic branches of dimension $h^0(S_0, \mathcal N_{S_0|\mathcal X_0})-1$, each branch corresponding to the locus of deformations of $S_0$ that are equisingular at a given point in $\mathfrak Z$, i.e.,  we have that   
 $$
 \mathcal T = \bigcap _{p\in \mathfrak Z} \mathcal {ES}_{[S_0],p}^{\mathcal X_0}.
 $$ 
 The general element $S'_0$ of $\mathcal {ES}_{[S_0],p}^{\mathcal X_0}$ is a surface in $\mathcal X_0$ that has a unique singularity analytically equivalent to the singularity of $S_0$ at $p$, i.e., a node if $p\in \mathfrak Z_A+\mathfrak Z_B$, a $T_1$ singularity otherwise. 
 
 By Proposition \ref {prop:defonodi}, for every $p\in \mathfrak Z_A+\mathfrak Z_B$,  $\mathcal {ES}_{[S_0],p}^{\mathcal X_0}$ is contained in $\mathcal {ES}_{[S_0],p}^{\mathcal X}$ as a subvariety of codimension 1, and $\mathcal {ES}_{[S_0],p}^{\mathcal X}$ is an analytic branch of the Severi variety $\mathcal V^{\mathcal X|\mathbb D}_1$. 
 
 If $p\in \mathfrak Z_R$, then by the hypothesis (b), for general element $S'_0$ of $\mathcal {ES}_{[S_0],p}^{\mathcal X_0}$ the condition (b) of Proposition \ref{lm:delta=1} holds. Then, by Proposition \ref{lm:delta=1}, $\mathcal {ES}_{[S_0],p}^{\mathcal X_0}$ is contained, as a subvariety of codimension 1, in an analytic branch
$\mathcal T_p$ of $\mathcal V_1^{\mathcal X|\mathbb D}$ having codimension $1$ in $\mathcal H$, which is smooth at the general point correponding to a $1$-nodal surface. 
 
  Now, the intersection
$$
\mathcal T'  = \bigcap _{p\in \mathfrak Z_R} \mathcal T_{p} \cap \bigcap _{p\in \mathfrak Z_A+\mathfrak Z_B} \mathcal {ES}_{[S_0],p}^{\mathcal X}
$$
has codimension at most $\delta$ in $\mathcal H$ and it contains the smooth analytic branch $\mathcal T$ of $\mathbb T$, which has codimension $\delta+1$ in $\mathcal H$.
The general element of 
$
 \mathcal T' 
$
corresponds to a surface $\tilde S$, not contained in $\mathcal X_0$, with at least $\delta$ singularities, precisely $\delta_A$ (resp. $\delta_B$) singularities in neighbohoods of the nodes $p\in \mathfrak Z_A$ (resp. $p\in \mathfrak Z_B$) and $\delta_R$ singularities in neighbohoods of the $T_1$ singularities  $p\in \mathfrak Z_R$.  Taking into account  Lemma \ref{lm:solo-nodo}, we deduce that 
$\tilde S$ has $\delta$ nodes and no further singularities. This proves (i). 

If $ \tilde S\subset\mathcal X_t$ and it has nodes at $\tilde q_1,...,\tilde q_\delta$, by semicontinuity, we have that
$H^0(\tilde S,\mathcal N_{\tilde S|\mathcal X_t}\otimes I_{\{\tilde q_1,...,\tilde q_\delta\}|\mathcal X_t})$ has dimension $\dim(\mathcal H_t)-\delta$. Thus
$
 \mathcal T'  
$ is an analytic branch containing the point $[\tilde S]$ in an irreducible component $\mathcal V\subset\mathcal V^{\mathcal X|\mathbb D}_\delta$ of dimension $\dim(\mathcal H)-\delta$ whose central fibre $\mathcal V_0$ contains $\mathbb T$ as an irreducible component. This proves (ii). 
\end{proof}


\section{Applications}\label{sec:appl}

\subsection{Severi varieties} Let $X$ be a smooth irreducible projective complex threefold. Let $L$ be a  line bundle on $X$
such that the general surface in the linear system $| L|$ is smooth and irreducible.
We denote by $V^{X, |L|}_\delta$ the  {\it Severi variety}, that is the locally closed  subscheme in $|L|$ parametrizing surfaces $S$ in $|L|$ which are reduced and with 
only $\delta$ nodes as singularities. If $[S]\in  V^{X, |L|}_\delta$, then the Zariski tangent space to 
$V^{X, |L|}_\delta$ at $[S]$ coincides with 
$$
T_{[S]}(V^{X, |L|}_\delta)\simeq H^0(S, \mathcal O_S(L)\otimes I_{N|S}),
$$
where $N$ is the reduced scheme of nodes of $S$. In particular, 
$
\dim(V^{X, |L|}_\delta)\leq h^0(S, \mathcal O_S(L)\otimes I_{N|S}).
$
Moreover, by standard deformation theory, $H^1(S, \mathcal O_S(L)\otimes I_{N|S})$ is an obstruction space for $\mathcal O_ {V^{X, |L|}_\delta,[S]}$
and thus 
$$
h^0(S, \mathcal O_S(L)\otimes I_{N|S})-h^1(S, \mathcal O_S(L)\otimes I_{N|S})
\leq \dim(V^{X, |L|}_\delta)\leq h^0(S, \mathcal O_S(L)\otimes I_{N|S}).
$$
If $h^1(S, \mathcal O_S(L)\otimes I_{N|S})=0$, then $V^{X, |L|}_\delta$ is smooth at $[S]$ of dimension 
$$
h^0(S, \mathcal O_S(L)\otimes I_{N|S})=\dim(|L|)-\delta.
$$
In this case one says that $[S]$ is a {\it regular} point of the $\dim(V^{X, |L|}_\delta)$. An irreducible component $V$ of 
$\dim(V^{X, |L|}_\delta)$ is said to be {\it regular} if it is regular at its general point. 

\begin{remark}\label{rem:reg} Suppose $V$ is a regular irreducible component of 
$\dim(V^{X, |L|}_\delta)$. By standard deformation theory already used in Section \ref {deformations},  the nodes of the surface corresponding to any smooth point in $V$ can be independently smoothed. This implies that there are regular components of $\dim(V^{X, |L|}_{\delta'})$ for any $\delta'<\delta$.
\end{remark}

One can consider the following two questions.
\begin{problem}\label{main}
Given $X$ and $L$ as above, which is the maximal value of $\delta$ such that 
the Severi variety $V^{X, |L|}_\delta$ is non empty?
\end{problem}

\begin{problem}\label{regular}
Given $X$ and $L$ as above, which is the maximal value of $\delta$ such that 
the Severi variety $V^{X, |L|}_\delta$ has a regular component?
\end{problem}
 
 As for Problem \ref{main}, this is a classical and difficult question, for which there are several contributions, too many to be quoted here, probably the most efficient one is  given by the Miyaoka's bound \cite [Formulae  (2) and  (8)] {Mi}. In particular the problem has been completely solved for  $X=\mathbb P^3$ and $L=\mathcal O_{\mathbb P^3}(d)$
with $d\leq 6$ (see, e.g.,  \cite {L} and references therein). 
 However, in this section we will not consider  Problem \ref{main} but we will give some contribution to Problem \ref{regular}.

\begin{remark} 
 One could be tempted to believe that the maximal $\delta$ for which the Severi variety is non-empty is bounded above
 by the dimension of $|L|$. This is  not true. In fact there are classical examples,  for $X=\mathbb P^3$ and $L=\mathcal O_{\mathbb P^3}(d)$
 for suitable $d$, for which $V^{X, |L|}_\delta$ is non-empty and $\delta$ is greater than the dimension of $|L|$ (cf. \cite{beauville}, \cite{segre}). 
  In these cases every component of the Severi variety is not regular.
 \end{remark}
 
 \begin{remark} \label{rm:massimo}
 Referring to Problem \ref{regular}, it is rather natural to conjecture that  the $\delta$ which answer the question  should be  bounded below by 
 $\delta_0=[{\dim(|L|)\over 4 }]$. The reason for such a conjecture is the following: choose $p_1,\ldots,p_{\delta_0}$ general points on $X$.
 Since a double point imposes at most 4 conditions to $|L|$, certainly there are surfaces which are singular at every $p_i$. If the general such surface
 has only nodes at $p_1,\ldots ,p_{\delta_0}$ and no other singularities then it belongs to a regular component of the Severi variety. However this heuristic argument is very difficult to be made rigorous in general.
 \end{remark}

 \subsection{The case of $\mathbb P^3$} In this section we give a contribution to Problem \ref {regular}, in the case $X=\mathbb P^3$ and $L=\mathcal O_{\mathbb P^3}(d)$. More precisely we will prove the following:
 
 \begin{theorem}\label{thm:reg} There is an irreducible, regular component of 
 $V^{\mathbb P^3, |\mathcal O_{\mathbb P^3}(d)|}_{\delta}$, for any $\delta \leq {{d-1}\choose 2}$.
 \end{theorem}
 
 \begin{proof} In view of Remark \ref{rem:reg},  it is sufficent to consider only the case $\delta={{d-1}\choose 2}.$
 
 Let $\mathcal X'=\mathbb P^3\times \mathbb D\to \mathbb D$ be a trivial family.  
 Let us consider $\mathcal  X\to\mathcal X'$  the blow up of a point $q$ in the central fibre $\mathbb P^3$ over $0\in \mathbb D$. Let $\mathcal X\to \mathbb D$ be the new family. The fibre over $t\in \mathbb D\setminus \{0\}$ of this family is $\mathcal X_t\cong \mathbb P^3$. The central fibre $\mathcal X_0$ consists of two components $A\cup B$, where $f:A\to \mathbb P^3$ is the blow--up of $\mathbb P^3$ at $q$, whereas $B\cong \mathbb P^3$ is the exceptional divisor in $\mathcal X$, and $A\cap B=R\cong \mathbb P^2$ is the exceptional divisor in $A$ and a plane in $B$. 
 
On $\mathcal X'$ there is a line bundle $\mathcal L'$, which is the pull--back via the first projection, of $\mathcal O_{\mathbb P^3}(d)$. We pull this back to $\mathcal X$ and  denote it $\mathcal L$. Now we consider on $\mathcal X$ the line bundle $\mathcal L\otimes \mathcal O_\mathcal X((1-d)B)$.  Its  restriction to the general fibre 
$\mathcal X_t$ is given by  $(\mathcal L\otimes \mathcal O_\mathcal X((1-d)B))|_{\mathcal X_t}\simeq\mathcal O_{\mathbb P^3}(d)$. 
As for the restriction of $\mathcal L\otimes \mathcal O_\mathcal X((1-d)B)$ to $\mathcal X_0$, we observe that  $(\mathcal L\otimes \mathcal O_\mathcal X((1-d)B))|_A\simeq\mathcal O_A(d)\otimes\mathcal O_A(-(d-1)R)$, where $\mathcal O_A(d)\simeq f^*(\mathcal O_{\mathbb P^3}(d))$; whereas $(\mathcal L\otimes \mathcal O_\mathcal X((1-d)B))|_B\simeq\mathcal O_{\mathbb P^3}(d-1)$ and, finally, the restriction of $\mathcal L\otimes \mathcal O_\mathcal X((1-d)B)$ to $R$ is $\mathcal O_{\mathbb P^2}(d-1)$.  One easily checks that the line bundle $\mathcal L\otimes \mathcal O_\mathcal X((1-d)B)$ verifies the hypotheses (1), (2) and (3) at the beginning of Section \ref{sect:smoothing-to-nodes}.

 We now consider on $R$ a curve $C$ which consists of the union of $d-1$ general lines. Its has $\delta={{d-1}\choose 2}$ nodes as singularities. By standard application of Bertini's theorem, there exists a smooth surface $S_B$ in $B$ of degree $d-1$ cutting out on $R$ the curve $C$. Similarly there exists a smooth surface 
 $S_A\in |\mathcal O_A(d)\otimes\mathcal O_A(-(d-1)R)|$ restricting to $C$ on $R$. To see this let $(x,y,z)$ be an affine coordinates system on $\mathbb P^3$ centered at $q$. If $\phi_1(x,y,z)=0$ is the equation of $C$ in the plane at infinity, and $\phi_2(x,y,z)$ is a general homogeneous polynomial of degree $d$ in $(x,y,z)$, then the projective closure $S_B$ of the degree $d$ affine surface with equation $\phi_1(x,y,z)+\phi_2(x,y,z)=0$ has a point of multiplicity $d-1$ at $q$ and no other singularities, and its minimal resolution obtained by blowing up $q$ is the required surface. 

Now $S_0=S_A\cup S_B$ is a Cartier divisor in $\mathcal X_0$ belonging to the linear system  $|\mathcal L\otimes \mathcal O_\mathcal X((1-d)B)|$. Moreover $S_0$ verifies all hypotheses of Theorem \ref{thm:main-theorem}. In particular, if $\mathfrak Z$ is the reduced scheme of the nodes of $C$, then $\mathfrak Z$ imposes independent conditions to $|\mathcal L\otimes \mathcal O_\mathcal X((1-d)B)|_R|=|\mathcal O_{\mathbb P^2}(d-1)|$,  and therefore to $|\mathcal L\otimes \mathcal O_\mathcal X((1-d)B)|$, because the Severi varieties of 
nodal curves in the plane are well known to be regular. By applying Theorem \ref{thm:main-theorem}, one may deform $S_0$ to a surface $S_t\subset\mathcal X_t$
with $\delta$ nodes and no further singularities, which are deformations of the $\delta$ singularities of type $T_1$ of $S_0$. Finally the nodes of $S_t$ impose independent conditions to surfaces of degree $d$ on $\mathcal X_t\simeq \mathbb P^3$. Hence $[S_t]\in V^{\mathbb P^3, |\mathcal O_{\mathbb P^3}(d)|}_{\delta}$ belongs to a regular component of the Severi variety.
\end{proof}

\begin{remark}
Taking into account Remark \ref{rm:massimo}, we believe that the previous results is far from being sharp, not even asymptotically. 
Indeed  we may expect that the bound on $\delta$ for the existence of regular components of the Severi variety of nodal surfaces
of degree $d$ in $\mathbb P^3$ could asymptotically go as $\delta\sim {d^3\over 24}$.
See also \cite[Corollary 4.1]{L} and related references for a very large upper  bound of the number of nodes $\delta$ of a surface in $\mathbb P^3$ in a regular
component of the Severi variety (if non-empty). Moreover our results could in principle be improved by imposing to $S_0$ nodes off $R$, but we do not dwell on this here.  
\end{remark}

\begin{remark} The known results about Problem \ref{regular} are very few. For example, in \cite{Kl} one proves that, {\it if} $V^{\mathbb P^3, |\mathcal O_{\mathbb P^3}(d)|}_\delta$ is non--empty, then every component of it is regular for $d\leq 7$ and for $d\geq 8$ and $\delta\leq 4d-5$, and this last bound is sharp (the case $d\leq 7$ was already proved in \cite {D}). Nonemptiness results for $d\leq 7$ are also well known (see, e.g.,  \cite[p. 120] {L}). 
In particular, our Theorem \ref {thm:reg} is, at the best of our knowledge, new as soon as $d\geq 8$. 
\end{remark}

 \subsection{Complete intersections in $\mathbb P^4$} In this Section we want to provide a partial answer to Problem \ref{regular} in the case of complete intersections in $\mathbb P^4$.
 
 Let $X$ be a general hypersurface of degree $h\geq 2$ in $\mathbb P^4$. We consider on $X$ the linear system $|\mathcal O_X(d)|$. Our aim is to construct regular components of $V^{X,|\mathcal O_X(d)|}_\delta$ with suitable $\delta$.
 
 \begin{theorem}\label{thm:regol} Let $d\geq h-1$ be an integer. 
 There are regular components of $V^{X,|\mathcal O_X(d)|}_\delta$ for 
 $$\delta\leq {{d+3}\choose 3}-{{d-h+1}\choose 3}-1.$$
 \end{theorem}
 
 \begin{proof}
 As usual, to prove the theorem, it suffices to do the case 
 $$\delta={{d+3}\choose 3}-{{d-h+1}\choose 3}-1.$$
 
 Let $Y$ be a general hypersurface of $\mathbb P^4$ of degree $h-1$ and $H$ be a general hyperplane, cutting $Y$ along a surface $R$,
 which is a general surface of degree $h-1$ in $H\simeq \mathbb P^3$. Let $X$  be  a general hypersurface of degree $h$ and let us consider the pencil
 generated by $X$ and $Y\cup H$. Specifically, if $X$ has equation $f=0$,  $Y$ has equation $g=0$ and $H$ has equation $\ell =0$, we will consider the 
 hypersurface $\mathcal X''$ in $\mathbb P^4\times \mathbb A^1 $, with equation $\{tf+g\ell =0,\,{\rm{with}}\,t\in \mathbb A^1\}$. Via the second projection $\mathcal X''\to\mathbb A^1$, this becomes a flat family of $3$-folds, with smooth general fibre $\mathcal X''_t$, corresponding to a general
 hypersurface of degree $h$ in $\mathbb P^4$, and  whose fibre over $0$ is $\mathcal X''_0=Y\cup H\subset\mathbb P^4$.
 We are interested to the singularities of $\mathcal X''$ in a neighbordhood of the central fibre, i.e., we are interested in what happens if $t$ belongs to a disc $\mathbb D$, centered at the origin. Thus we consider the family $$\mathcal X'=\{tf+g\ell =0,\,{ \rm with}\,t\in \mathbb D\}\to\mathbb D.$$ It is immediate to see
that the singular locus of $\mathcal X'$ coincides with the curve $D:t=f=g=\ell=0\subset \mathcal X'_0$, which is isomorphic  to a smooth complete intersection curve of type $(1,h-1,h)$ in $\mathbb P^4$ cut out on $R=Y\cap H$ by $X$.  Moreover, $\mathcal X'$ has double points along $D$ with tangent cone a quadric of rank 4. We resolve these singularities by blowing-up $\mathcal X'$ along  $D$. One obtains a new family $\tilde{\mathcal X}\to\mathbb D$ with the same general fibre as $\mathcal X'\to\mathbb D$
 and whose central fibre consists of three components $\tilde Y$ and $\tilde H$, the blow-ups of $Y$ and $H$ along $D$ and the exceptional divisor $\tilde{\Theta}$
 that is a $\mathbb P^1\times \mathbb P^1$ bundle over $D$. Now we can contract $ \tilde \Theta$ by contracting one of the two rulings of the  
 $\mathbb P^1\times \mathbb P^1$ bundle. We choose to do this in the direction of $Y$. We obtain a new family of 3-folds $\mathcal X\to\mathbb D$, with $\mathcal X$ smooth, 
 with fiber $\mathcal X_t=\mathcal X'_t$ over $t\neq 0$, and  whose central fiber $\mathcal X_0=A\cup B$, where now $B=H\simeq \mathbb P^3$ and $A={\rm Bl}_D(Y)$ is
 the blowing-up of  $Y$ along $D$ and $A$ and $B$ intersect transversally along a surface isomorphic to $R$, which we still denote by $R=A\cap B$. The exceptional divisor $\Theta$ in $A={\rm Bl}_D(Y)$ is a $\mathbb P^1$--bundle on $D\subset R$, intersecting $R$ along $D$. In particular $\Theta\simeq \mathbb P(\mathcal N_{D|Y})$. 
 
 Notice that one has a natural morphism $\tilde {\mathcal X}\to \mathbb P^4$. This  factors through a morphism $\phi:  \mathcal X\to \mathbb P^4$. The action of $\phi$ on $\mathcal X_0$ is as follows: it maps $B$ isomorphically to $H$ and it maps $A$ to $Y$ by contracting the exceptional divisor $\Theta$. Let us now set $\mathcal L_d= \phi^*(\mathcal O_{\mathbb P^4}(d))$ and assume that $d\geq h-1$. 
 
Recall  that $R$ is a general surface of degree $h-1$ in $\mathbb P^3$, with $h\geq 2$. 
By \cite {CC}, $V^{R, |\mathcal O_R(d)|}_\delta$ is non--empty and contains a regular component $V$ for 
$$
\delta=\dim (|\mathcal O_R(d)|)={{d+3}\choose 3}-{{d-h+1}\choose 3}-1.
$$
So we can choose a general curve $C$ in $V$, that is a complete intersection of type $(h-1,d)$ on $R$ with $\delta$ nodes. Using Bertini's theorem, we can assume that there is a divisor $S_0\in \mathcal L_{d|\mathcal X_0}$ that cuts out $C$ on $R$ and $S_0=S_A\cup S_B$ (the notation is obvious), with $S_A$ and $S_B$ smooth. 

Now $S_0$ verifies all hypotheses of Theorem \ref{thm:main-theorem}. In particular, if $\mathfrak Z$ is the reduced scheme of  nodes of $C$, then $\mathfrak Z$ imposes independent conditions to $\mathcal L_{d}|_{\mathcal X_0}$,  because the component $V$ of the Severi variety  is  regular. By applying Theorem \ref{thm:main-theorem}, one may deform $S_0$ to a surface $S_t\subset\mathcal X_t$
with $\delta$ nodes and no further singularities, which are deformations of the $\delta$ singularities of type $T_1$ of $S_0$. Finally the nodes of $S_t$ impose independent conditions to surfaces in $|\mathcal O_{X_t}(d)|$. Hence $[S_t]\in V^{\mathcal X_t, |\mathcal O_{\mathcal X_t}(d)|}_{\delta}$ belongs to a regular component of the Severi variety, as wanted. \end{proof}

{\bf Conflicts of interest}: none.

{\bf Acknowledgments and financial support}: The authors are grateful to the referee for careful reading. Moreover, the authors want to thank Th. Dedieu, B. Fantechi, R. Pardini, E. Sernesi, for useful discussions on the topics of this paper. The authors are member of GNSAGA of INdAM. In particular, 
 the second author acknowledges funding from the GNSAGA of INdAM and the European Union - NextGenerationEU under the National Recovery and Resilience Plan (PNRR) - Mission 4 Education and research - Component 2 From research to business - Investment 1.1, Prin 2022 "Geometry of algebraic structures: moduli, invariants, deformations", DD N. 104, 2/2/2022, proposal code 2022BTA242 - CUP J53D23003720006.
{}
\end{document}